\documentclass[11pt]{amsart}%
\usepackage{palatino, mathpazo}
\usepackage{amsfonts}
\usepackage{amsmath}
\usepackage{amssymb,latexsym}
\usepackage{graphicx}
\usepackage{amssymb}
\usepackage[mathscr]{eucal}
\usepackage{color}%
\providecommand{\U}[1]{\protect \rule{.1in}{.1in}}
\newtheorem{theorem}{Theorem}[section]

\newtheorem{corollary}[theorem]{Corollary}

\newtheorem{lemma}[theorem]{Lemma}

\newtheorem{proposition}[theorem]{Proposition}
\theoremstyle{remark}
\newtheorem{remark}[theorem]{Remark}

\numberwithin{equation}{section}

\begin{document}
\title[the spectral polynomial of Treibich-Verdier potential]{Real-root property of the spectral polynomial of the Treibich-Verdier
potential and related problems}
\author{Zhijie Chen}
\address{Department of Mathematical Sciences and Yau Mathematical Sciences Center,
Beijing, 100084, China }
\email{zjchen@math.tsinghua.edu.cn}
\author{Ting-Jung Kuo}
\address{Taida Institute for Mathematical Sciences (TIMS), National Taiwan University,
Taipei 10617, Taiwan }
\email{tjkuo1215@gmail.com}
\author{Chang-Shou Lin}
\address{Taida Institute for Mathematical Sciences (TIMS), Center for Advanced Study in
Theoretical Sciences (CASTS), National Taiwan University, Taipei 10617, Taiwan }
\email{cslin@math.ntu.edu.tw}
\author{Kouichi Takemura}
\address{School of Mathematics, University of Leeds, Leeds LS2 9JT, United Kingdom\\
Department of Mathematics, Faculty of Science and Engineering, Chuo
University, 1-13-27 Kasuga, Bunkyo-ku Tokyo 112-8551, Japan}
\email{takemura@math.chuo-u.ac.jp}

\begin{abstract}
We study the spectral polynomial of the Treibich-Verdier potential. Such
spectral polynomial, which is a generalization of the classical Lam\'{e}
polynomial, plays fundamental roles in both the finite-gap theory and the ODE
theory of Heun's equation. In this paper, we prove that all the roots of such
spectral polynomial are real and distinct under some assumptions. The proof
uses the classical concept of Sturm sequence and isomonodromic theories. We
also prove an analogous result for a polynomial associated with a generalized
Lam\'{e} equation. Differently, our new approach is based on the viewpoint of
the monodromy data.

\end{abstract}
\maketitle

\section{Introduction}

Throughout the paper, we use the notations $\omega_{0}=0$, $\omega_{1}=1$,
$\omega_{2}=\tau$, $\omega_{3}=1+\tau$ and $\Lambda_{\tau}=\mathbb{Z+Z}\tau$,
where $\tau \in \mathbb{H}=\{ \tau|\operatorname{Im}\tau>0\}$. Define $E_{\tau
}:=\mathbb{C}/\Lambda_{\tau}$ to be a flat torus in the plane and $E_{\tau
}[2]:=\{ \frac{\omega_{k}}{2}|0\leq k\leq3\}+\Lambda_{\tau}$ to be the set
consisting of the lattice points and 2-torsion points in $E_{\tau}$.

In the literature, a smooth period function $q(z)$ satisfying $q(z)\in
\mathbb{R}$ for $z\in \mathbb{R}$ is called a \emph{finite-gap potential} if
the set $\sigma_{b}(H)$ of $H=-d^{2}/dz^{2}+q(z)$ satisfies%
\[
\overline{\sigma_{b}(H)}\cap \mathbb{R}=[E_{0},E_{1}]\cup \lbrack E_{2}%
,E_{3}]\cup \cdot \cdot \cdot \cup \lbrack E_{2g},+\infty)
\]
with $E_{0}<E_{1}<\cdot \cdot \cdot<E_{2g}$, where $\sigma_{b}(H)$ is the
spectrum of bounded bands, that is,%
\[
E\in \sigma_{b}(H)\Leftrightarrow \text{ Every solution of }(H-E)f(z)=0\text{ is
bounded on }z\in \mathbb{R}\text{.}%
\]
Recall that $\wp(z)=\wp(z|\tau)$ is the Weierstrass elliptic function with
periods $1$ and $\tau$, defined by%
\[
\wp(z|\tau):=\frac{1}{z^{2}}+\sum_{\omega \in \Lambda_{\tau}\backslash
\{0\}}\left(  \frac{1}{(z-\omega)^{2}}-\frac{1}{\omega^{2}}\right)  ,
\]
and $e_{k}=e_{k}(\tau):=\wp(\frac{\omega_{k}}{2}|\tau)$ for $k\in \{1,2,3\}$.
In \cite{Ince}, Ince proved that if $\tau \in i\mathbb{R}_{>0}$ (i.e. $E_{\tau
}$ is a rectangle torus) and $n\in \mathbb{N}$, then the potential
$n(n+1)\wp(z+\frac{\omega_{k}}{2}|\tau)$, $k\in \{2,3\}$, is a finite-gap
potential. Surprisingly, the finite-gap potential is related to the KdV
theory. A potential $q(z)$ is called an \emph{algebro-geometric finite-gap
potential }if there is an odd-order differential operator $A=(d/dz)^{2m+1}%
+\sum_{j=0}^{2m-1}b_{j}(z)(d/dz)^{2m-1-j}$ such that $[A,-d^{2}/dz^{2}%
+q(z)]=0$, that is, $q(z)$ is a solution of KdV hierarchy equations (cf.
\cite{Cal}). Under the condition that $q(z)$ is smooth periodic and
real-valued on $\mathbb{R}$, it is known (cf. \cite{Gesztesy}) that $q(z)$ is
a finite-gap potential if and only if it is an algebro-geometric finite-gap potential.

In 1990's, Treibich and Verdier found a new algebro-geometric finite-gap
potential, which is now called the \emph{Treibich-Verdier potential
}(\cite{TV}). This potential could be written as%
\begin{equation}
q^{(l_{0},l_{1},l_{2},l_{3})}(z):=\sum_{k=0}^{3}l_{k}(l_{k}+1)\wp \left(
z+\tfrac{\omega_{k}}{2}\right)  , \label{qz}%
\end{equation}
where $l_{k}\in \mathbb{N}\cup \{0\}$ for all $k$. See
\cite{Smirnov,Tak1,Tak2,Tak3,Tak4,Tak5} and references therein for historical
reviews and subsequent developments. Notice that $E_{j}$ for the potential
(\ref{qz}) might not be real in general; see Remark \ref{remark3} below.

The polynomial $\prod_{i=0}^{2g}(E-E_{i})$ is called \emph{the spectral
polynomial} of the operator $H=-d^{2}/dz^{2}+q(z)$. A basic question is how to
determine the spectral polynomial of the operator%
\begin{equation}
H^{(l_{0},l_{1},l_{2},l_{3})}:=-\frac{d^{2}}{dz^{2}}+\sum_{k=0}^{3}l_{k}%
(l_{k}+1)\wp \left(  z+\tfrac{\omega_{k}}{2}\right)  . \label{oper}%
\end{equation}
Remark that $H^{(l_{0},l_{1},l_{2},l_{3})}$ is also the Hamiltonian of the
$BC_{1}$ (one particle) Inozemtsev model (cf. \cite{Tak1}).

In the literature, there are two methods to compute the spectral polynomials.
One way is to study the so-called "polynomial solutions" of the following
generalized Lam\'{e} equation (GLE):%
\begin{equation}
y^{\prime \prime}(z)=I(z)y(z)=\left[  \sum_{k=0}^{3}l_{k}(l_{k}+1)\wp \left(
z+\tfrac{\omega_{k}}{2}\right)  -E\right]  y(z)\text{ in }E_{\tau},
\label{GLE-0}%
\end{equation}
where $E\in \mathbb{C}$. GLE (\ref{GLE-0}) is a Fuchsian equation with singular
points in $E_{\tau}[2]$. By projecting $E_{\tau}$ onto $\mathbb{CP}^{1}$ via
$x=\wp(z)$, GLE (\ref{GLE-0}) becomes a second order Fuchsian equation with
four singular points $\{e_{1},e_{2},e_{3},\infty \}$ on $\mathbb{CP}^{1}$. See
(\ref{Heun}) in Section 3 for this new ODE. Therefore, GLE (\ref{GLE-0}) is an
elliptic form of Heun's equation. This fact was first pointed out by Darboux
\cite{Darboux} in 1882. Classically, people are interested in finding the
condition on $E$ such that the new ODE (\ref{Heun}) has a "polynomial"
solution in terms of $x$ (see Section 3 for a precise definition). It is known
that there is a \emph{polynomial}
\[
Q(E)=Q^{(l_{0},l_{1},l_{2},l_{3})}(E)=Q^{(l_{0},l_{1},l_{2},l_{3})}(E|\tau)
\]
such that the new ODE (\ref{Heun}) has a "polynomial" solution if and only if
$Q(E)=0$. Furthermore, it was proved in \cite[Section 6]{Tak5} that%
\begin{align}
&  \text{\textit{this polynomial $Q^{(l_{0},l_{1},l_{2},l_{3})}(E|\tau)$
coincides with the spectral}}\label{chara}\\
&  \text{\textit{polynomial of the operator $H^{(l_{0},l_{1},l_{2},l_{3})}$ up
to a multiplication.}}\nonumber
\end{align}
Therefore in this paper, we also denote by $Q^{(l_{0},l_{1},l_{2},l_{3}%
)}(E|\tau)$ to be the \textit{spectral polynomial} of the operator
$H^{(l_{0},l_{1},l_{2},l_{3})}$.

When $(l_{0},l_{1},l_{2},l_{3})=(n,0,0,0)$, GLE (\ref{GLE-0}) turns to be
\emph{the classical Lam\'{e} equation} (\cite{Halphen,Poole,Whittaker-Watson})%
\begin{equation}
y^{\prime \prime}(z)=[n(n+1)\wp(z)+B]y(z)\text{ \ in \ }E_{\tau}, \label{Lame}%
\end{equation}
and the corresponding polynomial $\ell_{n}(B|\tau):=Q^{(n,0,0,0)}(-B|\tau)$ is
called the \emph{Lam\'{e} polynomial} in the literature; see e.g.
\cite{Whittaker-Watson,Poole}. It is known that $\deg_{B}\ell_{n}=2n+1$. We
refer to \cite{Beukers-Waall,CLW} for the general theory of the Lam\'{e}
equation and explicit forms of the Lam\'{e} polynomial.

A remarkable result about $\ell_{n}(B|\tau)$ which is related to the
finite-gap phenomena is the following \emph{real-root property}.\medskip

\noindent \textbf{Theorem A.} (cf. \cite{Whittaker-Watson}) \emph{Let $\tau \in
i\mathbb{R}_{>0}$ and $n\in \mathbb{N}$. Then all the roots of $\ell_{n}%
(\cdot|\tau)$ are real and distinct}.\medskip

Theorem A has another important application. We associate the polynomial
$\ell_{n}(B|\tau)$ with a \emph{hyperelliptic curve Y$_{n}(\tau
):=\{(B,W)|W^{2}=\ell_{n}(B|\tau)\}.$} Since $\ell_{n}(B|\tau)\in
\mathbb{Q}[g_{2}(\tau),g_{3}(\tau)][B]$, where $g_{2}$, $g_{3}$ are
coefficients of%
\[
\wp^{\prime}(z|\tau)^{2}=4\wp(z|\tau)^{3}-g_{2}(\tau)\wp(z|\tau)-g_{3}(\tau),
\]
Theorem A implies that the discriminant of $\ell_{n}(\cdot|\tau)$, which
defines a modular form with respect to $SL(2,\mathbb{Z})$, has only finitely
many zeros modulo $SL(2,\mathbb{Z})$. This means that \emph{except for
finitely many tori, }$\ell_{n}(B|\tau)$\emph{ has no multiple roots}, or in
other words, \emph{the associated hyperelliptic curve Y$_{n}(\tau)$ is smooth
at its finite branch points}. See \cite[Theorem 7.4]{CLW}. This hyperelliptic
curve has some interesting geometric properties. For example, it was shown in
\cite{CLW} that $Y_{n}(\tau)\cup \{ \infty \}$ is a cover of the torus\emph{
}$E_{\tau}$.

The above argument highlights the importance of studying the following
question:\medskip

\noindent \textbf{(Q)}: \textit{For }$l_{0},l_{1},l_{2},l_{3}\in \mathbb{N}%
\cup \{0\}$\textit{ and }$\tau \in i\mathbb{R}_{>0}$,\textit{ whether are all
the roots of $Q^{(l_{0},l_{1},l_{2},l_{3})}(E|\tau)$ real and distinct}%
?\medskip

Theorem A already solves the special case $l_{1}=l_{2}=l_{3}=0$. Later, this
question was partially answered by the forth author \cite{Tak1}.\medskip

\noindent \textbf{Theorem B.} \cite[Proposition 3.3]{Tak1} \emph{Let }$\tau \in
i\mathbb{R}_{>0}$\emph{, }$l_{2}=l_{3}=0$ \emph{and} $l_{0},l_{1}\in
\mathbb{N}\cup \{0\}$.\emph{ Then all the roots of }$Q^{(l_{0},l_{1}%
,l_{2},l_{3})}(\cdot|\tau)$\emph{ are real and distinct}.\medskip

One main purpose of this paper is to settle question (Q) in new cases. Our
first main result is following.

\begin{theorem}
\label{thm1}Let $\tau \in i\mathbb{R}_{>0}$ and $l_{0},l_{1},l_{2},l_{3}%
\in \mathbb{N}\cup \{0\}$. Then all the roots of $Q^{(l_{0},l_{1},l_{2},l_{3}%
)}(\cdot|\tau)$ are real and distinct provided that one of the following
conditions holds.

\begin{itemize}
\item[(i)] $l_{3}=0$ and $l_{0}\geq l_{1}+l_{2}-1$.

\item[(ii)] $l_{0}+l_{3}+1=l_{1}+l_{2}$, $l_{2}+l_{3}\geq l_{0}+l_{1}+1$,
$l_{1}+l_{3}\geq l_{0}+l_{2}+1$.

\item[(iii)] $l_{0}+l_{3}=l_{1}+l_{2}+1$, $l_{0}+l_{1}\geq l_{2}+l_{3}+1$,
$l_{0}+l_{2}\geq l_{1}+l_{3}+1$.
\end{itemize}
\end{theorem}

We note that (i) is a generalization of Theorem B. In the case $l_{3}=0$,
Remark \ref{remark3} indicates that the condition $l_{0}\geq l_{1}+l_{2}-1$ is
sharp in the sense that $Q^{(2,2,2,0)}(\cdot|\tau)$ has two non-real roots
even for $\tau \in i\mathbb{R}_{>0}$. Our proof of Theorem \ref{thm1} will
apply the classical concept of Sturm sequence and some isomonodromic results
which were obtained via generalized Darboux transformations in \cite{Tak5}.
See Sections 2-3.

The second method to compute the polynomial $Q$ is to use the second symmetric
product of GLE (\ref{GLE-0}). This method is also well-known and closely
related to the monodromy representation of GLE (\ref{GLE-0}). See e.g.
\cite{Whittaker-Watson, Tak1,CLW}. Indeed, this method could work for a class
of ODEs including GLE (\ref{GLE-0}) and, particularly, its generalization:%
\begin{equation}
y^{\prime \prime}(z)=I^{(l_{0},l_{1},l_{2},l_{3})}(z;p,\tau)y(z)\text{ \ in
}E_{\tau}, \label{GLE-2}%
\end{equation}%
\begin{align*}
\text{with \  \  \ }I^{(l_{0},l_{1},l_{2},l_{3})}(z;p,\tau)=  &  q^{(l_{0}%
,l_{1},l_{2},l_{3})}(z)+\tfrac{3}{4}(\wp(z+p)+\wp(z-p))\\
&  +A(\zeta(z+p)-\zeta(z-p))+B,
\end{align*}
where $q^{(l_{0},l_{1},l_{2},l_{3})}$ is in (\ref{qz}), $A,B\in \mathbb{C}$ and
$\pm p\not \in E_{\tau}[2]$ are always assumed to be \emph{apparent
singularities} (i.e. non-logarithmic). Under this assumption, $B$ is
determined by $(p,A)$ as follows (see \cite{Chen-Kuo-Lin}):%
\begin{equation}
B=A^{2}-\zeta(2p)A-\tfrac{3}{4}\wp(2p)-\sum_{k=0}^{3}l_{k}(l_{k}+1)\wp \left(
p+\tfrac{\omega_{k}}{2}\right)  . \label{i60}%
\end{equation}
Here we recall that $\zeta(z)=\zeta(z|\tau):=-\int^{z}\wp(\xi|\tau)d\xi$ is
the Weierstrass zeta function with two quasi-periods:%
\begin{equation}
\eta_{1}(\tau)=\zeta(z+1|\tau)-\zeta(z|\tau),\text{ \ }\eta_{2}(\tau
)=\zeta(z+\tau|\tau)-\zeta(z|\tau). \label{quasi}%
\end{equation}
We are interested in GLE (\ref{GLE-2}) because it can be reduced to
(\ref{GLE-0}) with different values of $l_{k}$ by letting $p\rightarrow
\frac{\omega_{k}}{2}$ and has also a close relation to the well-known
Panlev\'{e} VI equation. For example, if $(A(\tau),B(\tau),p(\tau))$ depends
on $\tau$ suitably such that GLE (\ref{GLE-2}) preserves the monodromy as
$\tau$ deforms, then $p(\tau)$ satisfies the elliptic form of Panlev\'{e} VI
equation. See \cite{Chen-Kuo-Lin}.

By applying the second method to GLE (\ref{GLE-2}), we can obtain a
hyperelliptic curve $Y_{p}^{(l_{0},l_{1},l_{2},l_{3})}(\tau)=\{(A,W)|W^{2}%
=\mathcal{Q}^{(l_{0},l_{1},l_{2},l_{3})}(A;p,\tau)\}$ associated with GLE
(\ref{GLE-2}), where
\[
\mathcal{Q}^{(l_{0},l_{1},l_{2},l_{3})}(A;p,\tau)\in \mathbb{Q}[\wp(p|\tau
),\wp^{\prime}(p|\tau),e_{1}(\tau),e_{2}(\tau),e_{3}(\tau)][A]
\]
is a polynomial of $A$. Therefore, $\mathcal{Q}^{(l_{0},l_{1},l_{2},l_{3}%
)}(A;p,\tau)\in \mathbb{R}[A]$ if $\tau \in i\mathbb{R}_{>0}$ and $p\in
(0,\frac{1}{2})$. Similarly as $Y_{n}(\tau)$, $Y_{p}^{(l_{0},l_{1},l_{2}%
,l_{3})}(\tau)\cup \{ \infty \}$ has a natural covering over $E_{\tau}$ (see
\cite{CKL4}). Naturally we ask the following question: for fixed $p\in
(0,\frac{1}{2})$, is $Y_{p}^{(l_{0},l_{1},l_{2},l_{3})}(\tau)$ smooth except
for finitely many tori, or equivalently, does the polynomial $\mathcal{Q}%
^{(l_{0},l_{1},l_{2},l_{3})}(A;p,\tau)$ have distinct roots except for
finitely many tori?

This question seems not trivial because the form of $\mathcal{Q}^{(l_{0}%
,l_{1},l_{2},l_{3})}(A;p,\tau)$ is very complicated even for small $l_{k}$.
For $(l_{0},l_{1},l_{2},l_{3})=(1,0,0,0)$, we denote $y=A\wp^{\prime}(p)$,
$x=\wp(p)$ and write $\hat{\ell}_{1}(y;x,\tau)=\mathcal{Q}^{(1,0,0,0)}%
(A;p,\tau)$ for convenience. Then a calculation (see \cite{CKL4}) shows that
$\deg_{y}\hat{\ell}_{1}=6$ and {\allowdisplaybreaks{\footnotesize
\begin{align*}
\hspace*{-0.8cm}\hat{\ell}_{1} &  (y;x,\tau)=[{65536(4x^{3}-g_{2}x-g_{3})^{3}}]^{-1}\big[262144y^{6}-196608(12x^{2}-g_{2})y^{5}\\
&  +12288(400x^{4}-88g_{2}x^{2}+g_{2}^{2}-64g_{3}x)y^{4}+2048(8000x^{6}\\
&  -2512g_{2}x^{4}+380g_{2}^{2}x^{2}-7g_{2}^{3}-256g_{3}x^{3}+488g_{2}g_{3}x+320g_{3}^{2})y^{3}\\
&  -192(12x^{2}-g_{2})(40000x^{6}-17680g_{2}x^{4}+2028g_{2}^{2}x^{2}-3g_{2}^{3}\\
&  -15872g_{3}x^{3}+3712g_{2}g_{3}x+1792g_{3}^{2})y^{2}+16(9600000x^{10}\\
&  -7276800g_{2}x^{8}+1692032g_{2}^{2}x^{6}-134176g_{2}^{3}x^{4}+428g_{2}^{4}x^{2}+27g_{2}^{5}\\
&  -7741440g_{3}x^{7}+3287040g_{2}g_{3}x^{5}-386560g_{2}^{2}g_{3}x^{3}+2944g_{2}^{3}g_{3}x\\
&  +1308672g_{3}^{2}x^{4}-316416g_{2}g_{3}^{2}x^{2}+896g_{2}^{2}g_{3}^{2}-98304g_{3}^{3}x)y\\
&  -(8000x^{6}-3280g_{2}x^{4}-4g_{2}^{2}x^{2}+9g_{2}^{3}-4864g_{3}x^{3}+64g_{2}g_{3}x-256g_{3}^{2})\\
&  (11200x^{6}-6896g_{2}x^{4}+916g_{2}^{2}x^{2}+3g_{2}^{3}-8192g_{3}x^{3}+2048g_{2}g_{3}x+1024g_{3}^{2})\big].
\end{align*}
}}Our second result of this paper is following.

\begin{theorem}
\label{THM3}Let $\tau \in i\mathbb{R}_{>0}$. Then all the roots of $\hat{\ell
}_{1}(\cdot;x,\tau)=0$ are real and distinct whenever $x>e_{1}(\tau)$ or
$x<e_{2}(\tau)$. In other words,

\begin{itemize}
\item[(1)] if $\tau \in i\mathbb{R}_{>0}$ and $p\in(0,\frac{1}{2})$, then all
the roots of $\mathcal{Q}^{(1,0,0,0)}(\cdot;p,\tau)=0$ are real and distinct;

\item[(2)] if $\tau \in i\mathbb{R}_{>0}$ and $p\in(0,\frac{\tau}{2})$, then
all the roots of $\mathcal{Q}^{(1,0,0,0)}(\cdot;p,\tau)=0$ are purely
imaginary and distinct.
\end{itemize}
\end{theorem}

Here $(0,\frac{\tau}{2}):=\{s\tau|s\in(0,\frac{1}{2})\}$. The difference of
the assertions (1) and (2) comes from the well known fact that for $\tau \in
i\mathbb{R}_{>0}$, $\wp^{\prime}(p|\tau)\in \mathbb{R}$ if $p\in(0,\frac{1}%
{2})$ and $\wp^{\prime}(p|\tau)\in i\mathbb{R}$ if $p\in(0,\frac{\tau}{2})$.

Remark that due to the appearance of singularities $\pm p\not \in E_{\tau}%
[2]$, Theorem \ref{THM3} can not be proved via the same idea of Sturm sequence
as Theorem \ref{thm1}. Our new approach of proving Theorem \ref{THM3} contains
two steps. The first step is to write down the equation for the (not
completely reducible) monodromy data $C$'s of GLE (\ref{GLE-2}) with
$(l_{0},l_{1},l_{2},l_{3})=(1,0,0,0)$, and the second one is to prove that
such $C$'s are purely imaginary and distinct if $\tau \in i\mathbb{R}_{>0}$.
Since \emph{the monodromy of GLE (\ref{GLE-2}) is not completely reducible if
and only if} $\mathcal{Q}^{(l_{0},l_{1},l_{2},l_{3})}(A;p,\tau)=0$ (see
\cite{CKL4}), the equation of the monodromy data $C$'s is also a polynomial of
degree $6$. This new polynomial has some advantages: (i) It can be decomposed
as a product of four polynomials; (ii) It has a nice structure for each
factors in (i). Therefore, we do not need to know the explicit formula of
$\hat{\ell}_{1}(y;x,\tau)$ in the proof of Theorem \ref{THM3}.

The paper is organized as follows. Theorem \ref{thm1} will be proved in
Sections 2-3 and Theorem \ref{THM3} will be proved in Sections 4-5.

\section{Polynomial solutions and Sturm sequence}

The purpose of this and next sections is to prove Theorem \ref{thm1}. To apply
the idea of Sturm sequence, we need to investigate polynomial solutions of
\begin{equation}
\frac{d^{2}y}{dx^{2}}+\left(  \frac{\gamma_{1}}{x-t_{1}}+\frac{\gamma_{2}%
}{x-t_{2}}+\frac{\gamma_{3}}{x-t_{3}}\right)  \frac{dy}{dx}+\frac{\alpha
\beta(x-t_{3})-q}{\prod_{j=1}^{3}(x-t_{j})}y=0. \label{eq:Heunt1t2t3}%
\end{equation}
It is a Fuchsian equation on $\mathbb{CP}^{1}$ with four regular singularities
$\{t_{1},t_{2},t_{3},\infty \}$. We impose the condition $\alpha+\beta
+1=\gamma_{1}+\gamma_{2}+\gamma_{3}$ so that the exponents at $x=\infty$ are
$\alpha$ and $\beta$. Set
\begin{equation}
y=\sum_{m=0}^{\infty}c_{m}(x-t_{3})^{m},\quad \text{where}\;c_{0}=1,
\end{equation}
and substitute it to the differential equation which is multiplied by
$\prod_{j=1}^{3}(x-t_{j})$ to equation (\ref{eq:Heunt1t2t3}). Then the
coefficients satisfy the following recursive relations:
\begin{align}
(t_{1}-t_{3})(t_{2}-t_{3})\gamma_{3}c_{1}  &  =qc_{0},\label{eq:Hlci}\\
(t_{1}-t_{3})(t_{2}-t_{3})(m+1)(m+\gamma_{3})c_{m+1}  &  =-(m-1+\alpha
)(m-1+\beta)c_{m-1}\nonumber \\
+[m\{(m-1+\gamma_{3})(t_{1}+t_{2}-2t_{3})+  &  (t_{2}-t_{3})\gamma_{1}%
+(t_{1}-t_{3})\gamma_{2}\}+q]c_{m}.\nonumber
\end{align}
If $t_{1}\neq t_{2}\neq t_{3}\neq t_{1}$ and $\gamma_{3}\not \in -{\mathbb{Z}%
}_{\geq0}$, then it is easy to see that $c_{r}$ is a polynomial in $q$ of
degree $r$ and we denote it by $c_{r}(q)$.

Moreover we assume that $\alpha=-N$ or $\beta=-N$ for some $N \in{\mathbb{Z}%
}_{\geq0} $. Let $q_{0}$ be a solution to the equation $c_{N+1} (q)=0$. Then
it follows from (\ref{eq:Hlci}) for $m=N+1$ that $c_{N+2} (q_{0})=0$. By
applying (\ref{eq:Hlci}) for $m=N+2, N+3, \dots$, we have $c_{m} (q_{0})=0$
for $m\geq N+3$. Hence, if $c_{N+1} (q_{0})=0$, then the differential equation
(\ref{eq:Heunt1t2t3}) have a non-zero polynomial solution. More precisely, we
obtain the following proposition.

\begin{proposition}
\label{prop:Heunpolym} Assume that $t_{1} \neq t_{2} \neq t_{3} \neq t_{1}$,
$\gamma_{3} \not \in -{\mathbb{Z}}_{\geq0} $, $\{ \alpha, \beta \} =\{ -N,
\gamma_{1}+ \gamma_{2} +\gamma_{3} +N-1 \}$ and $N \in{\mathbb{Z}}_{\geq0} $.
If $q$ is a solution to the equation $c_{N+1} (q)=0$, then the differential
equation (\ref{eq:Heunt1t2t3}) have a non-zero polynomial solution of degree
no more than $N$.
\end{proposition}

Next, we restrict to the case that all the parameters are \emph{real}. The
following proposition is shown immediately by applying the recursive equation
(\ref{eq:Hlci}).

\begin{proposition}
\label{prop:alt}
Assume that $\{ \alpha, \beta \} =\{ -N, \gamma_{1}+ \gamma_{2} +\gamma_{3}
+N-1 \}$, $N \in{\mathbb{Z}}_{\geq0} $, $\gamma_{1}$, $\gamma_{2}$ and
$\gamma_{3}$ are real, $\gamma_{3} >0$, $\gamma_{1}+ \gamma_{2} +\gamma_{3} +N
>1$ and $(t_{1} -t_{3})(t_{2}- t_{3}) <0$. Then the followings hold:

\begin{itemize}
\item[$(i)$] The sign of the coefficient of $q^{m}$ in $c_{m}(q)$ is that of
$(-1)^{m}$.

\item[$(ii)$] If $1\leq m \leq N$ and $c_{m}(q)=0$, then the values of
$c_{m+1} (q)$ and $c_{m-1} (q)$ are opposite in sign.
\end{itemize}
\end{proposition}

By applying the argument of \emph{Sturm sequence}, we have:

\begin{theorem}
\label{thm:Sturm} Assume that $\{ \alpha, \beta \} =\{ -N, \gamma_{1}+
\gamma_{2} +\gamma_{3} +N-1 \}$, $N \in{\mathbb{Z}}_{\geq0} $, $\gamma_{1}$,
$\gamma_{2}$ and $\gamma_{3}$ are real, $\gamma_{3} >0$, $\gamma_{1}+
\gamma_{2} +\gamma_{3} +N >1$ and $(t_{1} -t_{3})(t_{2}- t_{3}) <0$. Then the
equation $c_{N+1} (q)=0$ has all its roots real and unequal.
\end{theorem}

\begin{proof}
We will show that the polynomial $c_{r}(q)$ $(1\leq r\leq N)$ has $r$ real
distinct roots $s_{i}^{(r)}$ $(i=1,\dots, r)$ such that
\[
s_{1}^{(r)}<s_{1}^{(r-1)}<s_{2}^{(r)}<s_{2}^{(r-1)}<\dots<s_{r-1}%
^{(r)}<s_{r-1}^{(r-1)}<s_{r}^{(r)}%
\]
by induction on $r$. The case $r=1$ is trivial. Let $k \in \mathbb{N}$ and
assume that the statement is true for $r\leq k$. From the assumption of the
induction,
\[
s_{1}^{(k)}<s_{1}^{(k-1)}<s_{2}^{(k)}<s_{2}^{(k-1)}<\dots<s_{k-1}%
^{(k-1)}<s_{k}^{(k)}.
\]
It follows from Proposition \ref{prop:alt} (i) that the sign of the
coefficient of $q^{r}$ in $c_{r}(q)$ is that of $(-1)^{r}$. Then
\[
\lim_{q \rightarrow-\infty} c_{k-1}(q) = +\infty, \; \lim_{q \rightarrow
+\infty} c_{k-1}(q) = (-1)^{k-1} \infty.
\]
For any $1\leq i\leq k-2$, since $c_{k-1}(s_{i}^{(k-1)})=c_{k-1}%
(s_{i+1}^{(k-1)})=0$ and $s_{i}^{(k-1)} <s_{i+1}^{(k)}< s_{i+1}^{(k-1)}$, the
sign of the value $c_{k-1}(s_{i+1}^{(k)})$ is $(-1)^{i}$. Furthermore, it
follows from $s^{(k)}_{1}<s^{(k-1)}_{1}$ and $s^{(k)}_{k}>s^{(k-1)}_{k-1}$
that $c_{k-1}(s^{(k)}_{1})>0$ and the sign of the value of $c_{k-1}%
(s^{(k)}_{k})$ is $(-1)^{k-1}$. In conclusion, the sign of the value
$c_{k-1}(s^{(k)}_{i})$ is $(-1)^{i-1}$ for any $1\leq i\leq k$. Then it
follows from Proposition \ref{prop:alt} (ii) that the sign of the value
$c_{k+1}(s_{i}^{(k)})$ is $(-1)^{i}$ for any $1\leq i\leq k$. Since
\begin{equation}
\lim_{q \rightarrow-\infty} c_{k+1}(q) = +\infty, \; \lim_{q \rightarrow
+\infty} c_{k+1}(q) = (-1)^{k+1} \infty,
\end{equation}
it follows from the intermediate value theorem that the polynomial
$c_{k+1}(E)$ has $k+1$ real distinct roots $s_{i}^{(k+1)} $ $(1\leq i\leq
k+1)$ such that the inequality $s_{1}^{(k+1)}<s_{1}^{(k)}<s_{2}^{(k+1)}%
<s_{2}^{(k)}<\dots<s_{k}^{(k)}<s_{k+1}^{(k+1)}$ is satisfied.
\end{proof}

If $t_{1}=t$, $t_{2}=1$ and $t_{3}=0$, then equation (\ref{eq:Heunt1t2t3}) is
the well-known \textit{Heun's equation}, which is a standard form of the
second order Fuchsian differential equation with four singularities on
$\mathbb{CP}^{1}$. Clearly the condition $(t_{1} -t_{3})(t_{2}- t_{3}) <0$ is
equivalent to $t<0$.

\section{Proof of Theorem \ref{thm1}}

In this section, we will apply Theorem \ref{thm:Sturm} to prove Theorem
\ref{thm1}. Recall the Hamiltonian of the $BC_{1}$ Inozemtsev model as
mentioned in Section 1:
\begin{equation}
H:=H^{(l_{0},l_{1},l_{2},l_{3})}:=-\frac{d^{2}}{dz^{2}}+\sum_{k=0}^{3}%
l_{k}(l_{k}+1)\wp(z+\tfrac{\omega_{k}}{2}). \label{Ino}%
\end{equation}
Note that the Hamiltonian is \emph{unchanged} by replacing $l_{k}$ to
$-l_{k}-1$ $(k=0,1,2,3)$.
Let $f(z)$ be an eigenfunction of $H$ with an eigenvalue $E$, i.e.
\begin{equation}
(H-E)f(z)=\left(  -\frac{d^{2}}{dz^{2}}+\sum_{k=0}^{3}l_{k}(l_{k}%
+1)\wp(z+\tfrac{\omega_{k}}{2})-E\right)  f(z)=0. \label{InoEF}%
\end{equation}
Set $x=\wp(z)$. Applying the formula
\[
\wp(z+\tfrac{\omega_{i}}{2})=e_{i}+\frac{(e_{i}-e_{i^{\prime}})(e_{i}%
-e_{i^{\prime \prime}})}{\wp(z)-e_{i}},\; \, \text{where $\{i,i^{\prime
},i^{\prime \prime}\}=\{1,2,3\}$,}
\]
it is easy to see that equation (\ref{InoEF}) is equivalent to
\begin{align}
&  \left \{  \frac{d^{2}}{dx^{2}}+\frac{1}{2}\left(  \frac{1}{x-e_{1}}+\frac
{1}{x-e_{2}}+\frac{1}{x-e_{3}}\right)  \frac{d}{dx}-\frac{1}{4\prod_{j=1}%
^{3}(x-e_{j})}\right. \label{Heun}\\
&  \; \left.  \cdot \left(  \tilde{C}+l_{0}(l_{0}+1)x+\sum_{i=1}^{3}l_{i}%
(l_{i}+1)\frac{(e_{i}-e_{i^{\prime}})(e_{i}-e_{i^{\prime \prime}})}{x-e_{i}%
}\right)  \right \}  \tilde{f}(x)=0,\nonumber
\end{align}
where $\tilde{f}(\wp(z))=f(z)$ and $\tilde{C}=-E+\sum_{i=1}^{3}l_{i}%
(l_{i}+1)e_{i}$. Note that $e_{1}+e_{2}+e_{3}=0$. It is easy to see that the
Riemann scheme of equation (\ref{Heun}) is
\[%
\begin{Bmatrix}
e_{1} & e_{2} & e_{3} & \infty \\
\frac{-l_{1}}{2} & \frac{-l_{2}}{2} & \frac{-l_{3}}{2} & \frac{-l_{0}}{2}\\
\frac{l_{1}+1}{2} & \frac{l_{2}+1}{2} & \frac{l_{3}+1}{2} & \frac{l_{0}+1}{2}%
\end{Bmatrix}
.
\]

Let $\tilde{\alpha}_{i}=-l_{i}/2$ or $(l_{i}+1)/2$ for each $i\in \{0,1,2,3\}$.
Set
\[
\Phi^{(\tilde{\alpha}_{1},\tilde{\alpha}_{2},\tilde{\alpha}_{3})}%
(x)=\prod_{j=1}^{3}(x-e_{j})^{\tilde{\alpha}_{j}}\quad \text{and}\quad \tilde
{f}(x)=\Phi^{(\tilde{\alpha}_{1},\tilde{\alpha}_{2},\tilde{\alpha}_{3}%
)}(x)f(x).
\]
Then $\tilde{f}(x)$ solves equation (\ref{Heun}) is equivalent to that $f(x)$
satisfies
\begin{align}
\frac{d^{2}f(x)}{dx^{2}}+  &  \sum_{i=1}^{3}\frac{2\tilde{\alpha}_{i}+\frac
{1}{2}}{x-e_{i}}\frac{df(x)}{dx}+\left(  \frac{(\tilde{\alpha}_{1}%
+\tilde{\alpha}_{2}+\tilde{\alpha}_{3}-\frac{l_{0}}{2})(\tilde{\alpha}%
_{1}+\tilde{\alpha}_{2}+\tilde{\alpha}_{3}+\frac{l_{0}+1}{2})x}{(x-e_{1}%
)(x-e_{2})(x-e_{3})}\right. \nonumber \\
&  \left.  +\frac{\frac{E}{4}-e_{1}(\tilde{\alpha}_{2}+\tilde{\alpha}_{3}%
)^{2}-e_{2}(\tilde{\alpha}_{1}+\tilde{\alpha}_{3})^{2}-e_{3}(\tilde{\alpha
}_{1}+\tilde{\alpha}_{2})^{2}}{(x-e_{1})(x-e_{2})(x-e_{3})}\right)  f(x)=0.
\label{Heun2}%
\end{align}
This equation is in the form of equation (\ref{eq:Heunt1t2t3}) by setting
\[
\alpha=\tilde{\alpha}_{0}+\tilde{\alpha}_{1}+\tilde{\alpha}_{2}+\tilde{\alpha
}_{3},
\]%
\[
\beta=-\tilde{\alpha}_{0}+\tfrac{1}{2}+\tilde{\alpha}_{1}+\tilde{\alpha}%
_{2}+\tilde{\alpha}_{3},
\]%
\[
\gamma_{i}=2\tilde{\alpha}_{i}+\tfrac{1}{2},\quad t_{i}=e_{i},\quad i=1,2,3,
\]%
\[
q=-\left(  \tfrac{E}{4}-e_{1}(\tilde{\alpha}_{2}+\tilde{\alpha}_{3})^{2}%
-e_{2}(\tilde{\alpha}_{1}+\tilde{\alpha}_{3})^{2}-e_{3}(\tilde{\alpha}%
_{1}+\tilde{\alpha}_{2})^{2}+e_{3}\alpha \beta \right)  .
\]
We assume $\gamma_{3}=2\tilde{\alpha}_{3}+\tfrac{1}{2}>0$ as before. It is
well known that $e_{j}=e_{j}(\tau)\in \mathbb{R}$ and $e_{1}>e_{3}>e_{2}$
provided $\tau \in i\mathbb{R}_{>0}$, i.e.
\begin{equation}
(e_{1}-e_{3})(e_{2}-e_{3})<0\quad \text{if}\; \tau \in i\mathbb{R}_{>0}.
\label{eqe}%
\end{equation}
Write $f(x)=\sum_{r=0}^{\infty}c_{r}(x-e_{3})^{r}$ with $c_{0}=1$, then
$c_{r}$ is a polynomial in $E$ of degree $r$. We denote it by $c_{r}(E)$. We
set $N=-\alpha=-\tilde{\alpha}_{0}-\tilde{\alpha}_{1}-\tilde{\alpha}%
_{2}-\tilde{\alpha}_{3}$ and assume $N\in{\mathbb{Z}}_{\geq0}$. Then the
propositions in the previous section hold true for $c_{r}(E)$. In particular,
it follows from Proposition \ref{prop:Heunpolym} that if $c_{N+1}(E)=0$, then
the differential equation (\ref{Heun}) has a "polynomial" solution $\tilde
{f}(x)=\Phi^{(\tilde{\alpha}_{1},\tilde{\alpha}_{2},\tilde{\alpha}_{3}%
)}(x)f(x)$ in the sense that $f(x)$ is a polynomial of degree no more than $N$.

Let $P_{\tilde{\alpha}_{0}, \tilde{\alpha}_{1}, \tilde{\alpha}_{2} ,
\tilde{\alpha}_{3}}(E)$ be \emph{the monic polynomial obtained by normalising
$c_{N+1} (E) $}. Then
\[
\deg P_{\tilde{\alpha}_{0}, \tilde{\alpha}_{1}, \tilde{\alpha}_{2} ,
\tilde{\alpha}_{3}}(E)= N+1= -\tilde{\alpha}_{0}-\tilde{\alpha}_{1}%
-\tilde{\alpha}_{2}-\tilde{\alpha}_{3}+1.
\]
By Theorem \ref{thm:Sturm} and (\ref{eqe}), we immediately obtain the
following theorem.

\begin{theorem}
\label{prop:ellipreal} Assume that $l_{0}, l_{1}, l_{2} , l_{3} \in
{\mathbb{R}}$, $N = -\tilde{\alpha}_{0}-\tilde{\alpha}_{1}-\tilde{\alpha}%
_{2}-\tilde{\alpha}_{3} \in{\mathbb{Z}}_{\geq0} $, $2 \tilde{\alpha}_{3} +1/2
>0$, $-\tilde{\alpha}_{0} + \tilde{\alpha}_{1} + \tilde{\alpha}_{2} +
\tilde{\alpha}_{3} +1/2>0$ and $\tau \in i\mathbb{R}_{>0}$. Then the equation
$P_{\tilde{\alpha}_{0}, \tilde{\alpha}_{1}, \tilde{\alpha}_{2} , \tilde
{\alpha}_{3}}(E)=0$ has all its roots real and unequal.
\end{theorem}

Remark that we do not need to assume $l_{j}\in \mathbb{Z}$ in Theorem
\ref{prop:ellipreal}.

From now on, we assume that $l_{0},l_{1},l_{2},l_{3}\in{\mathbb{Z}}_{\geq0}$.
We recall the following important result from \cite{Tak1}, which establishes
the precise relation between the spectral polynomial $Q^{(l_{0},l_{1}%
,l_{2},l_{3})}(E)$ and the aforementioned polynomial $P_{\tilde{\alpha}%
_{0},\tilde{\alpha}_{1},\tilde{\alpha}_{2},\tilde{\alpha}_{3}}(E)$. This plays
a key role in our proof of Theorem \ref{thm1}.

If $l_{0}+l_{1} +l_{2}+l_{3}$ is even, then the spectral polynomial $Q(E)
=Q^{(l_{0},l_{1},l_{2},l_{3})}(E)$ of $H^{(l_{0},l_{1},l_{2},l_{3})} $ is
written as $Q(E)=P^{(0)} (E) P^{(1)} (E) P^{(2)} (E) P^{(3)} (E)$,
where{\allowdisplaybreaks
\begin{align*}
&  P^{(0)} (E) = P_{-l_{0}/2, -l_{1}/2, -l_{2}/2, -l_{3}/2}(E),\\
&  P^{(1)} (E) = \left \{
\begin{array}
[c]{ll}%
P_{-l_{0}/2, -l_{1}/2, (l_{2} +1)/2, (l_{3} +1)/2}(E), & l_{0} +l_{1} \geq
l_{2} +l_{3} +2 ,\\
1 , & l_{0} +l_{1} = l_{2} +l_{3} ,\\
P_{(l_{0}+1)/2, (l_{1} +1)/2, -l_{2} /2, -l_{3} /2}(E), & l_{0} +l_{1} \leq
l_{2} +l_{3} -2,
\end{array}
\right. \\
&  P^{(2)} (E) = \left \{
\begin{array}
[c]{ll}%
P_{-l_{0}/2, (l_{1}+1)/2, -l_{2} /2, (l_{3} +1)/2}(E), & l_{0} +l_{2} \geq
l_{1} +l_{3} +2 ,\\
1 , & l_{0} +l_{2} = l_{1} +l_{3} ,\\
P_{(l_{0}+1)/2, -l_{1}/2, (l_{2} +1)/2, -l_{3} /2}(E), & l_{0} +l_{2} \leq
l_{1} +l_{3} -2,
\end{array}
\right. \\
&  P^{(3)} (E) = \left \{
\begin{array}
[c]{ll}%
P_{-l_{0}/2, (l_{1}+1)/2, (l_{2} +1)/2, -l_{3} /2}(E), & l_{0} +l_{3} \geq
l_{1} +l_{2} +2 ,\\
1 , & l_{0} +l_{3} = l_{1} +l_{2} ,\\
P_{(l_{0}+1)/2, -l_{1} /2, -l_{2} /2, (l_{3} +1)/2}(E), & l_{0} +l_{3} \leq
l_{1} +l_{2} -2.
\end{array}
\right.
\end{align*}
}If $l_{0}+l_{1} +l_{2}+l_{3}$ is odd, then the spectral polynomial $Q(E)$ is
written as $Q(E)=P^{(0)} (E) P^{(1)} (E) P^{(2)} (E) P^{(3)} (E)$,
where{\allowdisplaybreaks
\begin{align*}
&  P^{(0)} (E) =\left \{
\begin{array}
[c]{ll}%
P_{-l_{0}/2, (l_{1} +1)/2, (l_{2} +1)/2, (l_{3} +1)/2}(E), & l_{0} \geq l_{1}
+ l_{2} +l_{3} +3 ,\\
1 , & l_{0} = l_{1}+l_{2} +l_{3} +1,\\
P_{(l_{0}+1)/2, -l_{1} /2, -l_{2} /2, -l_{3} /2}(E), & l_{0} \leq l_{1} +l_{2}
+l_{3} -1,
\end{array}
\right. \\
&  P^{(1)} (E) = \left \{
\begin{array}
[c]{ll}%
P_{(l_{0}+1)/2, -l_{1}/2, (l_{2} +1)/2, (l_{3} +1)/2}(E), & l_{1} \geq l_{0} +
l_{2} +l_{3} +3 ,\\
1 , & l_{1} = l_{0}+l_{2} +l_{3} +1,\\
P_{-l_{0}/2, (l_{1} +1) /2, -l_{2} /2, -l_{3} /2}(E), & l_{1} \leq l_{0}
+l_{2} +l_{3} -1,
\end{array}
\right. \\
&  P^{(2)} (E) = \left \{
\begin{array}
[c]{ll}%
P_{(l_{0}+1)/2, (l_{1}+1)/2, -l_{2} /2, (l_{3} +1)/2}(E), & l_{2} \geq l_{0}
+l_{1} +l_{3} +3 ,\\
1 , & l_{2} = l_{0} +l_{1} +l_{3} +1,\\
P_{-l_{0}/2, -l_{1}/2, (l_{2} +1)/2, -l_{3} /2}(E), & l_{2} \leq l_{0} +l_{1}
+l_{3} -1,
\end{array}
\right. \\
&  P^{(3)} (E) = \left \{
\begin{array}
[c]{ll}%
P_{(l_{0} +1)/2, (l_{1}+1)/2, (l_{2} +1)/2, -l_{3} /2}(E), & l_{3} \geq l_{0}
+l_{1} +l_{2} +3 ,\\
1 , & l_{3} = l_{0} +l_{1} +l_{2} +1,\\
P_{-l_{0}/2, -l_{1} /2, -l_{2} /2, (l_{3} +1)/2}(E), & l_{3} \leq l_{0} +l_{1}
+l_{2} -1.
\end{array}
\right.
\end{align*}
}Furthermore, it was shown in \cite[Theorem 3.2]{Tak1} that the equations
$P^{(i)}(E)=0$ and $P^{(j)}(E)=0$ $(i\neq j)$ \textit{do not have common
solutions}. We sketch the proof here for the reader's convenience.

\begin{proof}
Assume that there exists a common solution $E=E_{0}$. Let $\tilde{f}^{(i)}(x)$
(resp. $\tilde{f}^{(j)}(x)$) be the solution of equation (\ref{Heun}) with
$E=E_{0}$. Since $\tilde{f}^{(i)}(\wp(z))$ and $\tilde{f}^{(j)}(\wp(z))$ form
a basis of solutions to (\ref{InoEF}), the Wronskian is a non-zero constant.
However it contradicts that the periodicity of $\tilde{f}^{(i)}(\wp(z))$ and
$\tilde{f}^{(j)}(\wp(z))$ with respect to the shift $z \to z+\omega_{1}$ or $z
\to z+\omega_{2}$ is different.
\end{proof}

The following result proves Theorem \ref{thm1} under the condition (i).

\begin{theorem}
\label{thm:Qrealdist} Assume that $l_{0}, l_{1}, l_{2} , l_{3} \in{\mathbb{Z}%
}_{\geq0}$. If $l_{3}=0$, $l_{0} \geq l_{1} +l_{2} -1$ and $\tau \in i
{\mathbb{R}}_{>0}$, then the zeros of the spectral polynomial $Q^{(l_{0}%
,l_{1},l_{2},l_{3})} (E)$ are all real and unequal.
\end{theorem}

\begin{proof}
We only need to show that the zeros of each polynomial $P^{(j)}(E)$,
$j\in \{0,1,2,3\}$, are all real and unequal.

First we consider the case that $l_{0} + l_{1} + l_{2} $ is even. Then $l_{0}
\geq l_{1} +l_{2} -1$ gives
\begin{equation}
\label{eqe1}l_{0} \geq l_{1} +l_{2}.
\end{equation}
For $P^{(0)}(E)$, we have
\[
(\tilde{\alpha}_{0} , \tilde{\alpha}_{1} , \tilde{\alpha}_{2} , \tilde{\alpha
}_{3} )= (-l_{0}/2, -l_{1}/2, -l_{2}/2, -l_{3}/2 ).
\]
Since $2\tilde{\alpha}_{3}+1/2=-l_{3} +1/2 =1/2>0$ and
\[
-\tilde{\alpha}_{0} + \tilde{\alpha}_{1} + \tilde{\alpha}_{2} + \tilde{\alpha
}_{3} +1/2=\frac{l_{0} - l_{1} - l_{2} +1}{2} >0,
\]
it follows from Theorem \ref{prop:ellipreal} that the zeros of $P^{(0)} (E) $
are all real and unequal. Note from (\ref{eqe1}) and $l_{3}=0$ that $l_{0}
+l_{1} \geq l_{2} +l_{3}$. Hence $P^{(1)} (E) = P_{-l_{0}/2, -l_{1}/2, (l_{2}
+1)/2, (l_{3} +1)/2}(E)$ or $1$. If $P^{(1)} (E)=P_{-l_{0}/2, -l_{1}/2, (l_{2}
+1)/2, (l_{3} +1)/2}(E)$, then $2\tilde{\alpha}_{3}+1/2=l_{3} +3/2 >0$ and
\[
-\tilde{\alpha}_{0} + \tilde{\alpha}_{1} + \tilde{\alpha}_{2} + \tilde{\alpha
}_{3} +1/2=\frac{l_{0} - l_{1} + l_{2} +3}{2} >0.
\]
Again Theorem \ref{prop:ellipreal} implies that the zeros of $P^{(1)} (E) $
are all real and unequal. It is shown similarly that the zeros of both
$P^{(2)} (E) $ and $P^{(3)} (E) $ are all real and unequal.

We consider the remaining case that $l_{0} + l_{1} + l_{2} $ is odd.

For $P^{(0)}(E)$, since $l_{0} \geq l_{1} +l_{2} -1$, we need to consider the
cases $l_{0} = l_{1} +l_{2} -1$ and $l_{0} \geq l_{1} +l_{2} +1$ separately.
If $l_{0} = l_{1} +l_{2} -1$, then
\[
P^{(0)} (E)= P_{(l_{0}+1)/2, -l_{1} /2, -l_{2} /2, -l_{3} /2}(E)
\quad \text{and so} \quad \deg P^{(0)} (E)=1.
\]
If $l_{0} \geq l_{1} + l_{2} +1 $, then $P^{(0)} (E) = P_{-l_{0}/2, (l_{1}
+1)/2, (l_{2} +1)/2, (l_{3} +1)/2}(E)$ or $1$. Again Theorem
\ref{prop:ellipreal} implies that the zeros of $P^{(0)} (E) $ are all real and
unequal. Clearly $l_{0} \geq l_{1} +l_{2} -1$ gives $l_{0} +l_{2} + 1 \geq
l_{1} +2 l_{2} \geq l_{1}$ and $l_{0} +l_{1} + 1\geq l_{2}$. Hence $P^{(1)}
(E) = P_{-l_{0}/2, (l_{1} +1) /2, -l_{2} /2, -l_{3} /2}(E)$ or $1$. Again it
follows from Theorem \ref{prop:ellipreal} that the zeros of $P^{(1)} (E) $ are
all real and unequal. It is shown similarly that the zeros of $P^{(2)} (E) $
and $P^{(3)} (E) $ are all real and unequal. The proof is complete.
\end{proof}

\begin{remark}
\label{remark3}If $l_{3}=0$ and $l_{0}\leq l_{1}+l_{2}-2$, then the spectral
polynomial $Q(E)$ may have non-real zeros in the case $\tau \in i\mathbb{R}%
_{>0}$. For example, in the case $l_{0}=l_{1}=l_{2}=2$ and $l_{3}=0$, the
spectral polynomial $Q(E)$ have non-real zeros in the case $\tau \in
i\mathbb{R}_{>0}$ where $e_{1}>e_{3}>e_{2}$. In fact, it was shown in \cite[p.
396 and 400]{Tak4} that the spectral polynomial $Q(E)$ in the case
$l_{0}=l_{1}=l_{2}=2$ and $l_{3}=0$ coincides with that in the case $l_{0}=3$
and $l_{1}=l_{2}=l_{3}=1$ and
\begin{align*}
&  Q(E)=P^{(0)}(E)P^{(1)}(E)P^{(2)}(E)P^{(3)}(E),\;P^{(k)}(E)=E-15e_{k}%
,\;(k=1,2,3),\\
&  P^{(0)}(E)=E^{4}-54g_{2}E^{2}-864g_{3}E-135g_{2}^{2}.
\end{align*}
Then two roots of $P^{(0)}(E)=0$ are not real in the case $e_{1}>e_{3}>e_{2}$.
We take the simplest case $\tau=i$ for example, where $g_{3}(i)=0$ gives
$P^{(0)}(E)=E^{4}-54g_{2}E^{2}-135g_{2}^{2}$, which clearly has two non-real
roots because $g_{2}(i)>0$.
\end{remark}

Next, we need to apply some isomonodromic results in \cite{Tak5} to prove
Theorem \ref{thm1} under the assumptions (ii)-(iii). Recall $l_{0},l_{1}%
,l_{2},l_{3}\in{\mathbb{Z}}_{\geq0}$. If $l_{0}+l_{1}+l_{2}+l_{3}$ is even, we
set
\begin{align}
&  l_{0}^{e}=(l_{0}-l_{1}-l_{2}-l_{3})/2-1,\quad l_{1}^{e}=(l_{0}-l_{1}%
+l_{2}+l_{3})/2,\label{def:lie}\\
&  l_{2}^{e}=(l_{0}+l_{1}-l_{2}+l_{3})/2,\quad l_{3}^{e}=(l_{0}+l_{1}%
+l_{2}-l_{3})/2.\nonumber
\end{align}
Note that $-1-l_{0}^{e}=(-l_{0}+l_{1}+l_{2}+l_{3})/2$. Then $l_{0}^{e}%
,l_{1}^{e},l_{2}^{e},l_{3}^{e}\in{\mathbb{Z}}$ and{\allowdisplaybreaks%
\begin{align}
&  l_{0}^{e}+l_{1}^{e}+l_{2}^{e}+l_{3}^{e}=2l_{0}-1,\quad l_{0}^{e}+l_{1}%
^{e}-l_{2}^{e}-l_{3}^{e}=-2l_{1}-1,\label{def:liecon}\\
&  l_{0}^{e}-l_{1}^{e}+l_{2}^{e}-l_{3}^{e}=-2l_{2}-1,\quad l_{0}^{e}-l_{1}%
^{e}-l_{2}^{e}+l_{3}^{e}=-2l_{3}-1.\nonumber
\end{align}
}By means of generalized Darboux transformations (i.e. for a suitable choice
of $(l_{0}^{\prime},l_{1}^{\prime},l_{2}^{\prime},l_{3}^{\prime})$, there
exists a differential operator $L$ with cofficients being ellipitc functions
such that $H^{(l_{0}^{\prime},l_{1}^{\prime},l_{2}^{\prime},l_{3}^{\prime}%
)}L=LH^{(l_{0},l_{1},l_{2},l_{3})}$), it was proved in \cite[Section 4]{Tak5}
that the following eight operators are isomonodromic:{\allowdisplaybreaks%
\begin{align}
&  H^{(l_{0}^{e},l_{1}^{e},l_{2}^{e},l_{3}^{e})},\;H^{(l_{1}^{e},l_{0}%
^{e},l_{3}^{e},l_{2}^{e})},\;H^{(l_{2}^{e},l_{3}^{e},l_{0}^{e},l_{1}^{e}%
)},\;H^{(l_{3}^{e},l_{2}^{e},l_{1}^{e},l_{0}^{e})},\label{eq:gDHe}\\
&  H^{(l_{0},l_{1},l_{2},l_{3})},\;H^{(l_{1},l_{0},l_{3},l_{2})}%
,\;H^{(l_{2},l_{3},l_{0},l_{1})},\;H^{(l_{3},l_{2},l_{1},l_{0})}.\nonumber
\end{align}
}Note that coincidence of the monodromy implies the coincidence of the
spectral polynomial $Q(E)$'s, because the zeros of the spectral polynomial is
characterized by the double periodicity up to signs of the eigenfunction
$\tilde{f}(\wp(z))$ (see \cite[Section 3]{Tak1}). Therefore we have
\begin{align}
&  Q^{(l_{0}^{e},l_{1}^{e},l_{2}^{e},l_{3}^{e})}(E)=Q^{(l_{1}^{e},l_{0}%
^{e},l_{3}^{e},l_{2}^{e})}(E)=Q^{(l_{2}^{e},l_{3}^{e},l_{0}^{e},l_{1}^{e}%
)}(E)=Q^{(l_{3}^{e},l_{2}^{e},l_{1}^{e},l_{0}^{e})}(E)\label{eq:gDQe}\\
&  =Q^{(l_{0},l_{1},l_{2},l_{3})}(E)=Q^{(l_{1},l_{0},l_{3},l_{2}%
)}(E)=Q^{(l_{2},l_{3},l_{0},l_{1})}(E)=Q^{(l_{3},l_{2},l_{1},l_{0}%
)}(E).\nonumber
\end{align}

If $l_{0} +l_{1} +l_{2} +l_{3}$ is odd, we set
\begin{align}
&  l_{0}^{o}= (l_{0}+l_{1}+l_{2}+l_{3}+1)/2, \quad l_{1}^{o}= (l_{0}%
+l_{1}-l_{2}-l_{3}-1)/2,\label{def:lio}\\
&  l_{2}^{o}= (l_{0}-l_{1}+l_{2}-l_{3}-1)/2, \quad l_{3}^{o}= (l_{0}%
-l_{1}-l_{2}+l_{3}-1)/2.\nonumber
\end{align}
Then $l_{0}^{o} , l_{1}^{o} , l_{2}^{o} , l_{3}^{o} \in{\mathbb{Z}}$ and
\begin{align}
&  l_{0}^{o} + l_{1}^{o} + l_{2}^{o} + l_{3}^{o} = 2 l_{0} - 1, \quad
l_{0}^{o} + l_{1}^{o} - l_{2}^{o} - l_{3}^{o} = 2 l_{1} + 1,
\label{def:liocon}\\
&  l_{0}^{o} - l_{1}^{o} + l_{2}^{o} - l_{3}^{o} = 2 l_{2} + 1, \quad
l_{0}^{o} - l_{1}^{o} - l_{2}^{o} + l_{3}^{o} = 2 l_{3} + 1.\nonumber
\end{align}
It follows from \cite[Section 4]{Tak5} that the following eight operators are
isomonodromic.
\begin{align}
&  H^{(l_{0}^{o},l_{1}^{o},l_{2}^{o},l_{3}^{o})}, \; H^{(l_{1}^{o},l_{0}%
^{o},l_{3}^{o},l_{2}^{o})}, \; H^{(l_{2}^{o},l_{3}^{o},l_{0}^{o},l_{1}^{o})},
\; H^{(l_{3}^{o},l_{2}^{o},l_{1}^{o},l_{0}^{o})},\label{eq:gDHo}\\
&  H^{(l_{0},l_{1},l_{2},l_{3})}, \; H^{(l_{1},l_{0},l_{3},l_{2})}, \;
H^{(l_{2},l_{3},l_{0},l_{1})}, \; H^{(l_{3},l_{2},l_{1},l_{0})}.\nonumber
\end{align}
Therefore we have
\begin{align}
&  Q^{(l_{0}^{o},l_{1}^{o},l_{2}^{o},l_{3}^{o})}(E)= Q^{(l_{1}^{o},l_{0}%
^{o},l_{3}^{o},l_{2}^{o})}(E)= Q^{(l_{2}^{o},l_{3}^{o},l_{0}^{o},l_{1}^{o}%
)}(E)= Q^{(l_{3}^{o},l_{2}^{o},l_{1}^{o},l_{0}^{o})}(E)\label{eq:gDQo}\\
&  =Q^{(l_{0},l_{1},l_{2},l_{3})}(E)= Q^{(l_{1},l_{0},l_{3},l_{2})} (E)=
Q^{(l_{2},l_{3},l_{0},l_{1})} (E)=Q^{(l_{3},l_{2},l_{1},l_{0})}(E) .\nonumber
\end{align}

We apply Theorem \ref{thm:Qrealdist} to Eqs.(\ref{eq:gDQe}, \ref{eq:gDQo}) and
immediately obtain

\begin{proposition}
\label{prop:llelo} Let $l_{3}=0$, $l_{0}, l_{1}, l_{2} \in{\mathbb{Z}}_{\geq
0}$, $l_{0} \geq l_{1} + l_{2} - 1 $ and $\tau \in i {\mathbb{R}}_{>0}$.

\begin{itemize}
\item[(i)] If $l_{0}+l_{1}+l_{2}$ is even, then the zeros of the spectral
polynomial $Q^{(l_{0}^{e},l_{1}^{e},l_{2}^{e},l_{3}^{e})} (E)$ are all real
and unequal, where $l_{0}^{e},l_{1}^{e},l_{2}^{e},l_{3}^{e}$ are defined by
Eq.(\ref{def:lie}).

\item[(ii)] If $l_{0}+l_{1}+l_{2}$ is odd, then the zeros of the spectral
polynomial $Q^{(l_{0}^{o},l_{1}^{o},l_{2}^{o},l_{3}^{o})} (E)$ are all real
and unequal, where $l_{0}^{o},l_{1}^{o},l_{2}^{o},l_{3}^{o}$ are defined by
Eq.(\ref{def:lio}).
\end{itemize}
\end{proposition}

In order to obtain the conditions (ii)-(iii) in Theorem \ref{thm1}, we need to
investigate what conditions $l_{0}^{e},l_{1}^{e},l_{2}^{e},l_{3}^{e}$ and
$l_{0}^{o},l_{1}^{o},l_{2}^{o},l_{3}^{o}$ satisfy in Proposition
\ref{prop:llelo}. We assume that $l_{3}=0$, $l_{0}, l_{1}, l_{2}
\in{\mathbb{Z}}_{\geq0}$ and $l_{0} \geq l_{1} + l_{2} - 1 $.

If $l_{0} +l_{1} +l_{2}$ is even, then $l_{0}\geq l_{1}+l_{2}$ and it follows
from $l_{3}=0$ and (\ref{def:liecon}) that $l_{0}^{e} -l_{1}^{e} - l_{2}^{e} +
l_{3}^{e} = - 1$, i.e.
\begin{equation}
\label{eqe2}l_{0}^{e} + l_{3}^{e} + 1 = l_{1}^{e} +l_{2}^{e}.
\end{equation}
It follows from $l_{0} \geq0$, $l_{1} \geq0$, $l_{2} \geq0$ and $l_{0} - l_{1}
- l_{2} \geq0 $ that
\begin{align}
\label{eqe3} &  l_{0}^{e} + l_{1}^{e} + l_{2}^{e} + l_{3}^{e} +1 \geq0, \quad
l_{0}^{e} + l_{1}^{e} - l_{2}^{e} - l_{3}^{e} +1 \leq0 ,\\
&  l_{0}^{e} - l_{1}^{e} + l_{2}^{e} - l_{3}^{e} +1 \leq0, \quad l_{0}^{e}=
(l_{0}-l_{1}-l_{2})/2-1 \geq-1,\nonumber
\end{align}
namely
\begin{align}
&  l_{0}^{e} + l_{1}^{e} + l_{2}^{e} + l_{3}^{e} \geq-1, \quad l_{2}^{e} +
l_{3}^{e} \geq l_{0}^{e} + l_{1}^{e} +1,\\
&  l_{1}^{e} + l_{3}^{e} \geq l_{0}^{e} + l_{2}^{e} +1 , \quad l_{0}^{e}
\geq-1 .\nonumber
\end{align}
Together with (\ref{eqe2}), we easily obtain
\begin{align}
&  2 l_{1}^{e} + l_{2}^{e} = l_{1}^{e} + l_{0}^{e} + l_{3}^{e} +1 \geq2
l_{0}^{e} + l_{2}^{e} +2,\\
&  2 l_{2}^{e} + l_{1}^{e} = l_{2}^{e} + l_{0}^{e} + l_{3}^{e} +1 \geq2
l_{0}^{e} + l_{1}^{e} +2,\nonumber \\
&  l_{1}^{e} + l_{2}^{e} +2 l_{3}^{e} \geq2 l_{0}^{e} + l_{1}^{e} + l_{2}^{e}
+2 ,\nonumber
\end{align}
so
\[
l_{1}^{e} \geq l_{0}^{e} +1 \geq0 , \; l_{2}^{e} \geq l_{0}^{e} +1 \geq0 , \;
l_{3}^{e} \geq l_{0}^{e} +1 \geq0 .
\]
Recall $l_{0}\geq l_{1}+l_{2}$. If $l_{0} >l_{1} +l_{2} +2 $, then
(\ref{eqe3}) gives $l_{0}^{e} > 0$ and so $l_{1}^{e} ,l_{2}^{e} , l_{3}^{e} >
0 $. If $l_{0} = l_{1} + l_{2} +2$, then $l_{0}^{e} = 0$ and $l_{1}^{e}
,l_{2}^{e} , l_{3}^{e} > 0 $. If $l_{0} = l_{1} + l_{2} $, then $l_{0}^{e} =
-1$ and
\[
l_{1}^{e}= \frac{l_{0}-l_{1}+l_{2}}{2} =l_{2},\; l_{2}^{e}= \frac{l_{0}%
+l_{1}-l_{2}}{2} =l_{1},\; l_{3}^{e}= \frac{l_{0}+l_{1}+l_{2}}{2} =l_{0}.
\]
In this case,
\[
H^{(l_{0}^{e},l_{1}^{e},l_{2}^{e},l_{3}^{e})} = H^{(-1,l_{2},l_{1},l_{0})}=
H^{(0,l_{2},l_{1},l_{0})}= H^{(l_{3},l_{2},l_{1},l_{0})},
\]
i.e. the operator $H^{(l_{0}^{e},l_{1}^{e},l_{2}^{e},l_{3}^{e})} $ coincides
with $H^{(l_{0},l_{1},l_{2},l_{3})} $ by shifting the variable as $x \to
x-\omega_{3}/2$.

Therefore, by rewriting $l_{0}^{e}, l_{1}^{e}, l_{2}^{e} , l_{3}^{e}$ by
$l_{0}, l_{1}, l_{2} , l_{3}$ and ignoring the case $l_{0}=-1$, we have the
following result.

\begin{theorem}
\label{thm2} Assume that $l_{0}, l_{1}, l_{2} , l_{3} \in{\mathbb{Z}}_{\geq0}%
$, $l_{0} +l_{3} +1 =l_{1} +l_{2} $, $l_{2} + l_{3} \geq l_{0} + l_{1} +1$,
$l_{1} + l_{3} \geq l_{0} + l_{2} +1$ and $\tau \in i {\mathbb{R}}_{>0}$. Then
the zeros of the spectral polynomial $Q^{(l_{0},l_{1},l_{2},l_{3})}(E)$ are
all real and unequal.
\end{theorem}

Now we turn to the case that $l_{0} +l_{1} +l_{2}$ is odd. Then it follows
from $l_{3}=0$ and (\ref{def:liocon}) that $l_{0}^{o} - l_{1}^{o} - l_{2}^{o}
+ l_{3}^{o} = 1$, i.e.
\begin{equation}
\label{eqe4}l_{0}^{o} + l_{3}^{o} = l_{1}^{o} +l_{2}^{o} +1.
\end{equation}
It follows from $l_{0} \geq0$, $l_{1} \geq0$, $l_{2} \geq0$ and $l_{0} - l_{1}
- l_{2} +1 \geq0 $ that
\begin{align}
\label{eqe5} &  l_{0}^{o} + l_{1}^{o} + l_{2}^{o} + l_{3}^{o} +1 \geq0, \quad
l_{0}^{o} + l_{1}^{o} - l_{2}^{o} - l_{3}^{o} -1 \geq0 ,\\
&  l_{0}^{o} - l_{1}^{o} + l_{2}^{o} - l_{3}^{o} -1 \geq0 , \quad l_{3}^{o}=
(l_{0}-l_{1}-l_{2}-1)/2 \geq-1,\nonumber
\end{align}
namely{\allowdisplaybreaks
\begin{align}
&  l_{0}^{o} + l_{1}^{o} + l_{2}^{o} + l_{3}^{o} \geq-1, \quad l_{0}^{o} +
l_{1}^{o} \geq l_{2}^{o} + l_{3}^{o} +1,\\
&  l_{0}^{o} + l_{2}^{o} \geq l_{1}^{o} + l_{3}^{o} +1 ,\quad l_{3}^{o} \geq-1
.\nonumber
\end{align}
}Together with (\ref{eqe4}), we obtain{\allowdisplaybreaks
\begin{align}
&  2 l_{1}^{o} + l_{2}^{o} = l_{1}^{o} + l_{0}^{o} + l_{3}^{o} -1 \geq
l_{2}^{o} + 2l_{3}^{o},\\
&  2 l_{2}^{o} + l_{1}^{o} = l_{2}^{o} + l_{0}^{o} + l_{3}^{o} -1 \geq2
l_{3}^{o} + l_{1}^{o} ,\nonumber \\
&  l_{1}^{o} + 2 l_{0}^{o} +l_{2}^{o} \geq l_{1}^{o} + l_{2}^{o} + 2l_{3}^{o}
+2 ,\nonumber
\end{align}
}so
\begin{align}
l_{0}^{o} \geq l_{3}^{o} +1 \geq0 , \; l_{1}^{o} \geq l_{3}^{o} \geq-1 , \;
l_{2}^{o} \geq l_{3}^{o} \geq-1 .
\end{align}
Recall $l_{0} \geq l_{1} + l_{2} - 1 $. If $l_{0} >l_{1} +l_{2} +1 $, then
(\ref{eqe5}) gives $l_{3}^{o} > 0$ and so $l_{0}^{o} ,l_{1}^{o} , l_{2}^{o} >
0 $. If $l_{0} = l_{1} + l_{2} +1$, then $l_{3}^{o} = 0=l_{3}$ and $l_{0}^{o}=
(l_{0}+l_{1}+l_{2}+1)/2=l_{0}$, $l_{1}^{o}= (l_{0}+l_{1}-l_{2}-1)/2= l_{1}$,
$l_{2}^{o}= (l_{0}-l_{1}+l_{2}-1)/2 = l_{2}$, i.e. $H^{(l_{0}^{o}, l_{1}^{o},
l_{2}^{o} , l_{3}^{o})}= H^{(l_{0},l_{1},l_{2},l_{3})}$. It remains to
consider the case $l_{0} = l_{1} + l_{2} -1$. Then $l_{3}^{o} = -1$ and
$l_{0}^{o}=l_{0} +1$, $l_{1}^{o}= l_{1} -1$, $l_{2}^{o}= l_{2} -1$, i.e.
$H^{(l_{0}^{o}, l_{1}^{o}, l_{2}^{o} , l_{3}^{o})}= H^{(l_{0} +1,l_{1}-1
,l_{2}-1,-1)}$. If $l_{1} >0$ and $l_{2} >0$, then
\[
H^{(l_{0}^{o}, l_{1}^{o}, l_{2}^{o} , l_{3}^{o})}=H^{(l_{0} +1,l_{1}-1
,l_{2}-1,0)}%
\]
reduces to the case in Theorem \ref{thm:Qrealdist}; if $l_{1}=0$, then
\[
H^{(l_{0}^{o}, l_{1}^{o}, l_{2}^{o} , l_{3}^{o})}= H^{(l_{0} +1,-1
,l_{2}-1,-1)}=H^{(l_{2},0 ,l_{0},0)}%
\]
reduces to the case in Theorem \ref{thm:Qrealdist} too; if $l_{2}=0$, then
\[
H^{(l_{0}^{o}, l_{1}^{o}, l_{2}^{o} , l_{3}^{o})}= H^{(l_{0} +1,l_{1}-1
,-1,-1)}=H^{(l_{1},l_{0},0,0)}%
\]
reduces to the case in Theorem \ref{thm:Qrealdist} or Theorem B.

Therefore, by rewriting $l_{0}^{o}, l_{1}^{o}, l_{2}^{o} , l_{3}^{o}$ by
$l_{0}, l_{1}, l_{2} , l_{3}$ and ignoring the case $l_{0} = l_{1} + l_{3}
\pm1 $, we have the following result.

\begin{theorem}
\label{thm3} Assume that $l_{0}, l_{1}, l_{2} , l_{3} \in{\mathbb{Z}}_{\geq0}%
$, $l_{0} +l_{3} =l_{1} +l_{2} +1$, $l_{0} + l_{1} \geq l_{2} + l_{3} +1$,
$l_{0} + l_{2} \geq l_{1} + l_{3} +1 $ and $\tau \in i {\mathbb{R}}_{>0}$. Then
the zeros of the spectral polynomial $Q^{(l_{0},l_{1},l_{2},l_{3})}(E)$ are
all real and unequal.
\end{theorem}

We are in the position to prove Theorem \ref{thm1}.

\begin{proof}
[Proof of Theorem \ref{thm1}]Theorem \ref{thm1} follows directly from Theorems
\ref{thm:Qrealdist}, \ref{thm2} and \ref{thm3}.
\end{proof}

\section{From the viewpoint of the monodromy data}

\label{RTS-PVI}

The purpose of this and next sections is to prove Theorem \ref{THM3} from the
viewpoint of the monodromy data of GLE (\ref{GLE-2}). There are two ways to
discuss the monodromy of GLE (\ref{GLE-2}). One way is to project GLE
(\ref{GLE-2}) to a new equation on $\mathbb{CP}^{1}$ via $x=\wp(z)$, which is
a second order Fuchsian equation with five singular points $\{e_{1}%
,e_{2},e_{3},\wp(p),\infty \}$ with $\wp(p)$ being apparent. Since a vast
literature has been devoted to studying second order linear ODEs defined on
$\mathbb{CP}^{1}$, we could apply some known theories to this new ODE.
However, the explicit formulas for the monodromy of this new ODE are not easy
to compute. Therefore, it is more convenient for us to calculate the monodromy
of GLE (\ref{GLE-2}) on the torus $E_{\tau}$ directly.

The monodromy representation of GLE (\ref{GLE-2}) is a homomorphism $\rho
:\pi_{1}\left(  E_{\tau}\backslash(E_{\tau}[2]\cup(\left \{  \pm p\right \}
+\Lambda_{\tau}),q_{0}\right)  \rightarrow SL(2,\mathbb{C})$, where
$q_{0}\not \in E_{\tau}[2]\cup(\left \{  \pm p\right \}  +\Lambda_{\tau})$ is a
base point. Let $\gamma_{\pm}\in \pi_{1}\left(  E_{\tau}\backslash(E_{\tau
}[2]\cup(\left \{  \pm p\right \}  +\Lambda_{\tau})),q_{0}\right)  $ be a simple
loop encircling $\pm p$ counterclockwise respectively, and $\ell_{j}$,
$j=1,2$, be two fundamental cycles of $E_{\tau}$ connecting $q_{0}$ with
$q_{0}+\omega_{j}$ such that $\ell_{j}$ does not intersect with $L+\Lambda
_{\tau}$ (here $L$ is the straight segment connecting $\pm p$) and satisfies%
\begin{equation}
\gamma_{+}\gamma_{-}=\ell_{1}\ell_{2}\ell_{1}^{-1}\ell_{2}^{-1}\text{ in }%
\pi_{1}\left(  E_{\tau}\backslash(\left \{  \pm p\right \}  +\Lambda_{\tau
}),q_{0}\right)  . \label{II-iv}%
\end{equation}
Since the local exponents of (\ref{GLE-2}) at $\pm p$ are $-\frac{1}{2}$ and
$\frac{3}{2}$ and $\pm p\not \in E_{\tau}[2]$ are apparent singularities, we
always have%
\begin{equation}
\rho(\gamma_{\pm})=-I_{2}. \label{89-2}%
\end{equation}
For any $k\in \{0,1,2,3\}$, the local exponents of GLE (\ref{GLE-2}) at
$\omega_{k}/2$ are $-l_{k}$ and $l_{k}+1$ with $l_{k}\in \mathbb{N}\cup \{0\}$.
Since the potential $I^{(l_{0},l_{1},l_{2},l_{3})}(z;p,\tau)$ is even
elliptic, the local monodromy matrix of GLE (\ref{GLE-2}) at $\omega_{k}/2$ is
$I_{2}$ (see e.g. \cite[Lemma 2.2]{Takemura}). Therefore, the monodromy group
of GLE (\ref{GLE-2}) is generated by $\{-I_{2},\rho(\ell_{1}),\rho(\ell
_{2})\}$. Together with (\ref{II-iv}) and (\ref{89-2}), we immediately obtain
$\rho(\ell_{1})\rho(\ell_{2})=\rho(\ell_{2})\rho(\ell_{1})$, which implies
that the monodromy group of GLE (\ref{GLE-2}) is always \emph{abelian} and
hence \emph{reducible}, i.e. \emph{all the monodromy matrices have at least a
common eigenfunction}. Clearly there are two cases:

\begin{itemize}
\item[(a)] Completely reducible, i.e. all the monodromy matrices have two
linearly independent common eigenfunctions: Up to a common conjugation,
$\rho(\ell_{1})$ and $\rho(\ell_{2})$ can be diagonalized as%
\[
\rho(\ell_{1})=%
\begin{pmatrix}
e^{-2\pi is} & 0\\
0 & e^{2\pi is}%
\end{pmatrix}
,\text{ \  \  \ }\rho(\ell_{2})=%
\begin{pmatrix}
e^{2\pi ir} & 0\\
0 & e^{-2\pi ir}%
\end{pmatrix}
\]
for some $(r,s)\in \mathbb{C}^{2}\backslash \frac{1}{2}\mathbb{Z}^{2}$.
Moreover, by using these two common eigenfunctions, we could express their
Wronskian $W$ in terms of $(A,p)$ as $W^{2}=\mathcal{Q}^{(l_{0},l_{1}%
,l_{2},l_{3})}(A;p,\tau)$, where $\mathcal{Q}^{(l_{0},l_{1},l_{2},l_{3}%
)}(A;p,\tau)$ is precisely the polynomial introduced in Section 1; see
\cite{CKL4}. So this case occurs if and only if $\mathcal{Q}^{(l_{0}%
,l_{1},l_{2},l_{3})}(A;p,\tau)\not =0$. This procedure is to use the second
symmetric product of GLE (\ref{GLE-2}) and has become well known in the
literature; see \cite{Whittaker-Watson,Takemura,Tak1} for example.

\item[(b)] Not completely reducible, i.e. the space of common eigenfunctions
is of dimension $1$: Up to a common conjugation,%
\begin{equation}
\rho(\ell_{1})=\varepsilon_{1}%
\begin{pmatrix}
1 & 0\\
1 & 1
\end{pmatrix}
,\text{ \  \  \ }\rho(\ell_{2})=\varepsilon_{2}%
\begin{pmatrix}
1 & 0\\
C & 1
\end{pmatrix}
, \label{Mono-2}%
\end{equation}
where $\varepsilon_{1},\varepsilon_{2}\in \{ \pm1\}$ and $C\in \mathbb{C}\cup \{
\infty \}$. Remark that if $C=\infty$, then (\ref{Mono-2}) should be understood
as%
\[
\rho(\ell_{1})=\varepsilon_{1}%
\begin{pmatrix}
1 & 0\\
0 & 1
\end{pmatrix}
,\text{ \  \  \ }\rho(\ell_{2})=\varepsilon_{2}%
\begin{pmatrix}
1 & 0\\
1 & 1
\end{pmatrix}
.
\]
In this case, $C$ is called the \emph{monodromy data} of GLE (\ref{GLE-2}).
Clearly this case occurs if and only if $\mathcal{Q}^{(l_{0},l_{1},l_{2}%
,l_{3})}(A;p,\tau)=0$.
\end{itemize}

\begin{remark}
As discussed before, GLE (\ref{GLE-2}) can be projected to a new Fuchsian ODE
on $\mathbb{CP}^{1}$. Then the monodromy representation of this new ODE is
irreducible if and only if Case (a) occurs, and reducible if and only if Case
(b) occurs. Most of the references in the literature are devoted to the case
of irreducible representation on $\mathbb{CP}^{1}$, but very few are devoted
to studying reducible representation.
\end{remark}

From now on, we consider the special case $(l_{0},l_{1},l_{2},l_{3}%
)=(1,0,0,0)$ and denote $\mathcal{Q}(A;p,\tau)=\mathcal{Q}^{(1,0,0,0)}%
(A;p,\tau)$ for convenience. Then GLE (\ref{GLE-2}) becomes%
\begin{equation}
y^{\prime \prime}(z)=\left[
\begin{array}
[c]{l}%
2\wp(z)+\frac{3}{4}(\wp(z+p)+\wp(z-p))\\
+A(\zeta(z+p)-\zeta(z-p))+B
\end{array}
\right]  y(z)\text{ \ in }E_{\tau}. \label{89-1}%
\end{equation}
Now suppose $\mathcal{Q}(A;p,\tau)=0$. Then as explained above, Case (b)
occurs with some monodromy data $C\in \mathbb{C}\cup \{ \infty \}$. The following
result is to express $\wp(p|\tau)$ in terms of this $C$.\medskip

\noindent \textbf{Theorem 4.A. }\cite{CKL3} \emph{Fix }$\tau \in \mathbb{H}%
$\emph{ and }$p\not \in E_{\tau}[2]$\emph{.}

\begin{itemize}
\item[(1)] \emph{If the monodromy of GLE (\ref{89-1}) is not completely
reducible (i.e. }$\mathcal{Q}(A;p,\tau)=0$\emph{)}, \emph{then the monodromy
data }$C$\emph{ satisfies either}%
\begin{equation}
\wp(p|\tau)=\frac{-4(C\eta_{1}-\eta_{2})^{3}-g_{2}(C\eta_{1}-\eta_{2}%
)(C-\tau)^{2}+2g_{3}(C-\tau)^{3}}{(C-\tau)[12(C\eta_{1}-\eta_{2})^{2}%
-g_{2}(C-\tau)^{2}]}, \label{III-17}%
\end{equation}
\emph{with} $(\varepsilon_{1},\varepsilon_{2})=(1,1)$ \emph{or}%
\begin{equation}
\wp(p|\tau)=\frac{(\frac{g_{2}}{2}-3e_{k}^{2})(C\eta_{1}-\eta_{2})+\frac
{g_{2}}{4}e_{k}(C-\tau)}{3e_{k}(C\eta_{1}-\eta_{2})+(\frac{g_{2}}{2}%
-3e_{k}^{2})(C-\tau)}, \label{III-18}%
\end{equation}
\emph{for some }$k\in \{1,2,3\}$\emph{ with}%
\begin{equation}
(\varepsilon_{1},\varepsilon_{2})=\left \{
\begin{array}
[c]{l}%
(1,-1)\text{ \ if \ }k=1,\\
(-1,1)\text{ \ if \ }k=2,\\
(-1,-1)\text{ \ if \ }k=3.
\end{array}
\right.  \label{III-19}%
\end{equation}
\emph{Conversely, the following holds.}

\item[(2-i)] \emph{If }$C\in \mathbb{C}\cup \{ \infty \}$\emph{ satisfies the
cubic equation (\ref{III-17})}, \emph{then there exists }$A\in \mathbb{C}%
$\emph{ (and }$B$\emph{ is given by }$(p,A)$\emph{ via (\ref{i60}) with
}$(l_{0},l_{1},l_{2},l_{3})=(1,0,0,0)$\emph{) such that for the corresponding
GLE (\ref{89-1}), up to a common conjugation,}%
\[
\rho(\ell_{1})=%
\begin{pmatrix}
1 & 0\\
1 & 1
\end{pmatrix}
,\text{ \  \  \ }\rho(\ell_{2})=%
\begin{pmatrix}
1 & 0\\
C & 1
\end{pmatrix}
.
\]

\item[(2-ii)] \emph{Fix }$k\in \{1,2,3\}$\emph{. If }$C\in \mathbb{C}\cup \{
\infty \}$\emph{ satisfies the equation (\ref{III-18}), then there exists
}$A\in \mathbb{C}$\emph{ such that for the corresponding GLE (\ref{89-1}), up
to a common conjugation,}%
\[
\rho(\ell_{1})=\varepsilon_{1}%
\begin{pmatrix}
1 & 0\\
1 & 1
\end{pmatrix}
,\text{ \  \  \ }\rho(\ell_{2})=\varepsilon_{2}%
\begin{pmatrix}
1 & 0\\
C & 1
\end{pmatrix}
,
\]
\emph{where }$(\varepsilon_{1},\varepsilon_{2})$ \emph{is given by
(\ref{III-19})}.
\end{itemize}

\begin{remark}
The formulas (\ref{III-17})-(\ref{III-18}) first appeared in
\cite[(3.68)-(3.69)]{Takemura} without detailed proofs and later was obtained
in \cite{CKL2} independently, as explicit expressions of Riccati type
solutions of Painlev\'{e} VI equation. But their connection with the monodromy
data seems not be well addressed. The assertion (1) can be proved directly
without applying Painlev\'{e} VI equation. However, the converse part of the
assertion (1), i.e. the assertion (2), is very delicate. Recall $\deg
_{A}\mathcal{Q}=6$. If $\mathcal{Q}(A;p,\tau)$ has six distinct roots, then
the assertion (2) may follow from the assertion (1) and the $1$-$1$
correspondence $A\mapsto(\varepsilon_{1},\varepsilon_{2},C)$. However, if
$\mathcal{Q}(A;p,\tau)$ has multiple roots, then it is not clear whether any
$C$ satisfying either (\ref{III-17}) or (\ref{III-18}) is the monodromy data
of some GLE (\ref{89-1}) or not. Without knowing the root structure of
$\mathcal{Q}(A;p,\tau)$, the proof of the assertion (2) is to apply the
connection between GLE (\ref{89-1}) and Painlev\'{e} VI equation. Since the
proof of Theorem 4.A is long and has nothing related to Theorem \ref{THM3}, we
refer the proof to \cite{CKL3}.
\end{remark}

How to apply Theorem 4.A to obtain Theorem \ref{THM3}? For $k\in \{1,2,3\}$, we
denote $C_{k}$ to be the unique solution of equation (\ref{III-18}). Then by
Theorem 4.A, there exists $A_{k}\in \mathbb{C}$ such that the monodromy representation
of the corresponding GLE (\ref{89-1}) is not completely reducible, i.e.
$\mathcal{Q}(A_{k};p,\tau)=0$. Similarly, we let $C_{j}$, $j\in \{4,5,6\}$ be
the three solutions of the cubic equation\emph{ }(\ref{III-17}). Again by
Theorem 4.A, there exists $A_{j}\in \mathbb{C}$ such that $\mathcal{Q}(A_{j};p,\tau
)=0$. Clearly%
\[
A_{k}\not =A_{j}\text{ \ for any }k\in \{1,2,3\} \text{ and }j\in \{4,5,6\},
\]%
\[
A_{k_{1}}\not =A_{k_{2}}\text{ for any }k_{1}\not =k_{2}\in \{1,2,3\},
\]
because the $(\varepsilon_{1},\varepsilon_{2})$'s of the corresponding GLEs
(\ref{89-1}) are different. Consequently, once we can prove that the cubic
equation\emph{ }(\ref{III-17}) has three distinct solutions, i.e.
\begin{equation}
C_{4}\not =C_{5}\not =C_{6}\not =C_{4}, \label{iv}%
\end{equation}
then we immediately obtain%
\[
A_{k}\not =A_{j}\text{ \ for any }k\not =\text{ }j\in \{1,2,3,4,5,6\},
\]
which will give all the roots of $\mathcal{Q}(A;p,\tau)=0$ because $\deg
_{A}\mathcal{Q}=6$. Therefore,

\begin{corollary}
Fix $\tau \in \mathbb{H}$ and $p\not \in E_{\tau}[2]$. Then $\mathcal{Q}%
(A;p,\tau)=0$ has six distinct roots if and only if the cubic equation\emph{
}(\ref{III-17}) has three distinct solutions.
\end{corollary}

In our case $\tau \in i\mathbb{R}_{>0}$ and $p\in(0,\frac{1}{2})\cup
(0,\frac{\tau}{2})$, we just need to prove (\ref{iv}) and $A_{j}\in \mathbb{R}$
for $p\in(0,\frac{1}{2})$ (resp. $A_{j}\in i\mathbb{R}$ for $p\in(0,\frac
{\tau}{2})$) for all $j$ to obtain Theorem \ref{THM3}. The full details will
be given in the next section.

\section{Proof of Theorem \ref{THM3}}

This section is devoted to the proof of Theorem \ref{THM3}. First we prove the
following result.

\begin{lemma}
\label{lemm}Let $\tau \in i\mathbb{R}_{>0}$ and $j\in \{1,2,3,4,5,6\}$. Then the
followings hold.

\begin{itemize}
\item[(1)] If $p\in(0,\frac{1}{2})$, then $C_{j}\in i\mathbb{R\cup \{ \infty
\}}$ if and only if $A_{j}\in \mathbb{R}$.

\item[(2)] If $p\in(0,\frac{\tau}{2})$, then $C_{j}\in i\mathbb{R\cup \{
\infty \}}$ if and only if $A_{j}\in i\mathbb{R}$.
\end{itemize}

\begin{proof}
Recall $C_{j}$ is the monodromy data of the GLE (\ref{89-1}):%
\begin{equation}
y^{\prime \prime}(z)=I(z)y(z)\text{ in }E_{\tau}, \label{GLE-3}%
\end{equation}
where%
\begin{equation}
I(z)=\left[
\begin{array}
[c]{l}%
2\wp(z|\tau)+\frac{3}{4}(\wp(z+p|\tau)+\wp(z-p|\tau))\\
+A_{j}(\zeta(z+p|\tau)-\zeta(z-p|\tau))+B_{j}%
\end{array}
\right]  , \label{GLE-30}%
\end{equation}
and $B_{j}$\emph{ }is given by $(p,A_{j})$ via (\ref{i60}) with\emph{ }%
$(l_{0},l_{1},l_{2},l_{3})=(1,0,0,0)$.

(1) Since $p\in(0,\frac{1}{2})$, we have $\bar{p}=p$. Let%
\[
I_{0}(z):=\left[
\begin{array}
[c]{l}%
2\wp(z|\tau)+\frac{3}{4}(\wp(z+p|\tau)+\wp(z-p|\tau))\\
+\bar{A}_{j}(\zeta(z+p|\tau)-\zeta(z-p|\tau))+\bar{B}_{j}%
\end{array}
\right]  .
\]
Since $\tau \in i\mathbb{R}_{>0}$, we easily obtain $\overline{I(\bar{z}%
)}=I_{0}(z)$. Consequently, if $y(z)$ is a local solution of GLE
(\ref{GLE-3})-(\ref{GLE-30}) in a small domain bounded away from $\pm
p+\Lambda_{\tau}$, then%
\begin{equation}
\tilde{y}(z):=\overline{y(\bar{z})} \label{eqe6}%
\end{equation}
is a local solution of GLE $y^{\prime \prime}=I_{0}(z)y(z)$. By (\ref{eqe6}),
it is easy to prove that the monodromy representation of GLE $y^{\prime \prime
}=I_{0}(z)y(z)$ has the same $(\varepsilon_{1},\varepsilon_{2})$ as that of
GLE (\ref{GLE-3})-(\ref{GLE-30}) and the monodromy data is $-\bar{C}_{j}$.

If $A_{j}\in \mathbb{R}$, i.e. $\bar{A}_{j}=A_{j}$, then $I_{0}(z)=I(z)$,
namely GLE $y^{\prime \prime}=I_{0}(z)y(z)$ coincides with GLE (\ref{GLE-3}%
)-(\ref{GLE-30}). Consequently, the monodromy data $-\bar{C}_{j}=C_{j}$, i.e.
$C_{j}\in i\mathbb{R}\cup \{ \infty \}$.

Conversely, if $C_{j}\in i\mathbb{R}\cup \{ \infty \}$, we have $-\bar{C}%
_{j}=C_{j}$, namely GLE $y^{\prime \prime}=I_{0}(z)y(z)$ has the same monodromy
group generators $\rho(\ell_{1})$, $\rho(\ell_{2})$ as GLE (\ref{GLE-3}%
)-(\ref{GLE-30}). Applying a uniqueness result of such GLE with respect to the
monodromy group generators $\rho(\ell_{1})$, $\rho(\ell_{2})$ (see
\cite{CKL4}), we conclude that $I_{0}(z)=I(z)$, which gives $\bar{A}_{j}%
=A_{j}$, i.e. $A_{j}\in \mathbb{R}$.

(2) Since $p\in(0,\frac{\tau}{2})$, i.e. $p$ is purely imaginary, we have
$\bar{p}=-p$. Define%
\[
\tilde{I}_{0}(z):=\left[
\begin{array}
[c]{l}%
2\wp(z|\tau)+\frac{3}{4}(\wp(z+p|\tau)+\wp(z-p|\tau))\\
-\bar{A}_{j}(\zeta(z+p|\tau)-\zeta(z-p|\tau))+\bar{B}_{j}%
\end{array}
\right]  .
\]
Then $\overline{I(\bar{z})}=\tilde{I}_{0}(z)$. The same argument as (1)
implies that $C_{j}\in i\mathbb{R}\cup \{ \infty \}$ if and only if $\tilde
{I}_{0}(z)=I(z)$, i.e. $-\bar{A}_{j}=A_{j}$, which is just $A_{j}\in
i\mathbb{R}$.
\end{proof}
\end{lemma}

Next, we need to prove the following result.

\begin{theorem}
\label{THM-13}Let $\tau \in i\mathbb{R}_{>0}$ and $p\in(0,\frac{1}{2}%
]\cup(0,\frac{\tau}{2}]$. Then the three roots $C$'s of equation%
\begin{equation}
\wp(p|\tau)=\frac{-4(C\eta_{1}-\eta_{2})^{3}-g_{2}(C\eta_{1}-\eta_{2}%
)(C-\tau)^{2}+2g_{3}(C-\tau)^{3}}{(C-\tau)[12(C\eta_{1}-\eta_{2})^{2}%
-g_{2}(C-\tau)^{2}]} \label{i-68}%
\end{equation}
are distinct and all belong to $i\mathbb{R}\cup \{ \infty \}$.
\end{theorem}

Before we go to prove Theorem \ref{THM-13}, we are in the position to prove
Theorem \ref{THM3}.

\begin{proof}
[Proof of Theorem \ref{THM3}]Fix $\tau \in i\mathbb{R}_{>0}$ and $p\in
(0,\frac{1}{2})\cup(0,\frac{\tau}{2})$. It is well known that%
\[
e_{1},e_{2},e_{3},g_{2},\eta_{1},\wp(p)\in \mathbb{R}\text{ \ and \ }\eta
_{2}\in i\mathbb{R}\text{.}%
\]
Then it follows from equation (\ref{III-18}) that $C_{k}\in i\mathbb{R}\cup \{
\infty \}$ for all $k\in \{1,2,3\}$. Together with the argument at the end of
Section 4, it follows from Lemma \ref{lemm} and Theorem \ref{THM-13} that
$A_{j}\in \mathbb{R}$ for $p\in(0,\frac{1}{2})$ (resp. $A_{j}\in i\mathbb{R}$
for $p\in(0,\frac{\tau}{2})$) and are all distinct for $j\in \{1,2,3,4,5,6\}$.
Therefore, $\mathcal{Q}(A;p,\tau)=0$ has six distinct real roots for
$p\in(0,\frac{1}{2})$ and six distinct purely imaginary roots for
$p\in(0,\frac{\tau}{2})$. Finally, by denoting $y_{j}=A_{j}\wp^{\prime}%
(p|\tau)$, we conclude from $\wp^{\prime}(p|\tau)\in \mathbb{R}$ for
$p\in(0,\frac{1}{2})$ and $\wp^{\prime}(p|\tau)\in i\mathbb{R}$ for
$p\in(0,\frac{\tau}{2})$ that such $y_{j}$'s are real and give all the roots
of $\hat{\ell}_{1}(y;x,\tau)=0$, so $\hat{\ell}_{1}(y;x,\tau)=0$ has six
distinct real roots. The proof is complete.
\end{proof}

We will exploit a conceptual idea to prove Theorem \ref{THM-13}. Remark that
$C=\tau$ can not be a root of equation (\ref{i-68}) because $\wp(p|\tau
)\not =\infty$. Denote
\[
X=\frac{C\eta_{1}(\tau)-\eta_{2}(\tau)}{\tau-C}\text{ \ and }x=\wp(p|\tau).
\]
Then equation (\ref{i-68}) is equivalent to%
\begin{equation}
X^{3}-3xX^{2}+\frac{g_{2}}{4}X+\frac{2g_{3}+xg_{2}}{4}=0. \label{i-69}%
\end{equation}
Since $\tau \in i\mathbb{R}_{>0}$, it is well known that
\[
g_{2}(\tau)>0,\text{ }g_{3}(\tau)\in \mathbb{R},\text{ }x=\wp(p|\tau)\left \{
\begin{array}
[c]{c}%
\geq e_{1}(\tau)>0\text{ \ if }p\in(0,\frac{1}{2}],\\
\leq e_{2}(\tau)<0\text{ \ if }p\in(0,\frac{\tau}{2}].
\end{array}
\right.
\]
On the other hand, a straightforward computation implies that the discriminant
of equation (\ref{i-69}) is%
\[
\Delta=\varphi(x;\tau)/16,
\]
where%
\[
\varphi(x;\tau):=432g_{2}x^{4}+864g_{3}x^{3}-72g_{2}^{2}x^{2}-216g_{2}%
g_{3}x-g_{2}^{3}-108g_{3}^{2}.
\]

\begin{lemma}
\label{lemma-16}Let $\tau \in i\mathbb{R}_{>0}$ and $p\in(0,\frac{1}{2}%
]\cup(0,\frac{\tau}{2}]$. Then equation (\ref{i-69}) has three real distinct
roots $X$'s provided that one of the following conditions hold:

\begin{itemize}
\item[(1)] $p\in(0,\frac{1}{2}]$ and $\tau=ib$ with $b\geq1$.

\item[(2)] for any fixed $\tau \in i\mathbb{R}_{>0}$, either $p>0$ is
sufficiently small or $\frac{1}{2}-p\geq0$ is sufficiently small.

\item[(3)] $p\in(0,\frac{\tau}{2}]$ and $\tau=ib$ with $b\in(0,1]$.

\item[(4)] for any fixed $\tau \in i\mathbb{R}_{>0}$, $p\in(0,\frac{\tau}{2}]$
satisfies either $|p|$ or $|\frac{\tau}{2}-p|$ is sufficiently small.
\end{itemize}
\end{lemma}

\begin{proof}
It is known that equation (\ref{i-69}) has three real distinct roots if and
only if the discriminant $\Delta>0$, i.e. $\varphi(x;\tau)>0$.

(1)-(2). Recall $p\in(0,\frac{1}{2}]$ gives $x\geq e_{1}$. Since $\tau=ib$
with $b\geq1$, so $e_{1}>0\geq e_{3}>e_{2}$, i.e. $g_{3}=4e_{1}e_{2}e_{3}%
\geq0$ and%
\[
0<g_{2}=4(e_{1}^{2}-e_{2}e_{3})\leq4e_{1}^{2}\leq4x^{2}.
\]
Consequently,
\begin{align*}
\varphi(x;\tau)  &  =72g_{2}x^{2}(6x^{2}-g_{2})+216g_{3}x(4x^{2}-g_{2}%
)-g_{2}^{3}-108g_{3}^{2}\\
&  \geq72g_{2}\times \frac{g_{2}}{4}\times \frac{g_{2}}{2}-g_{2}^{3}%
-108g_{3}^{2}\\
&  =8g_{2}^{3}-108g_{3}^{2}>0,
\end{align*}
where $g_{2}^{3}-27g_{3}^{2}>0$ is used in the last inequality. This proves (1).

To prove (2), we fix any $\tau=ib$ with $b\in(0,1)$. Then $e_{1}>e_{3}%
>0>e_{2}$. Denote%
\[
\alpha=e_{1}^{2}\text{ and }\beta=-e_{2}e_{3}=(e_{1}+e_{3})e_{3}<2\alpha.
\]
Then $g_{2}=4(\alpha+\beta)$ and $e_{1}g_{3}=-4\alpha \beta<0$. Consequently, a
direct computation leads to%
\begin{align*}
\varphi(e_{1};\tau)  &  =432g_{2}e_{1}^{4}+864g_{3}e_{1}^{3}-72g_{2}^{2}%
e_{1}^{2}-216g_{2}g_{3}e_{1}-g_{2}^{3}-108g_{3}^{2}\\
&  =64(2\alpha-\beta)^{3}>0.
\end{align*}
Therefore, $\varphi(x;\tau)>0$ if $x-e_{1}(\tau)\geq0$ is sufficiently small,
namely provided that $\frac{1}{2}-p\geq0$ is sufficiently small. Finally, when
$p>0$ is sufficiently small, then $x=\wp(p|\tau)$ is sufficiently large, which
clearly implies $\varphi(x;\tau)>0$. This proves (2).

(3)-(4). Recall $p\in(0,\frac{\tau}{2}]$ gives $x\leq e_{2}<0$. Since
$xg_{3}\geq0$, $4x^{2}\geq g_{2}$ for $\tau=ib$ with $b\in(0,1]$ and
$-e_{1}e_{3}<2e_{2}^{2}$ for $\tau=ib$ with $b>1$, the proof is the same as
that of (1)-(2) by exchanging the roles of $e_{1}$ and $e_{2}$.
\end{proof}

\begin{lemma}
\label{lemma-17}Let $\tau \in i\mathbb{R}_{>0}$ and $p\in(0,\frac{1}{2}%
]\cup(0,\frac{\tau}{2}]$. Then equation (\ref{i-69}) has three real distinct
roots $X$'s.
\end{lemma}

\begin{proof}
First we consider $p\in(0,\frac{1}{2}]$. Instead of proving $\varphi
(x;\tau)>0$ when $\tau=ib$ and $b\in(0,1)$ (which seems non-trivial because
$g_{3}<0$), here we exploit a conceptual proof. Define%
\[
b_{0}=\inf \left \{  b_{1}>0\left \vert
\begin{array}
[c]{c}%
\text{(\ref{i-69}) has three real distinct roots}\\
\text{for }\tau=ib\text{ with }b>b_{1}\text{ and }p\in(0,\frac{1}{2}]
\end{array}
\right.  \right \}  .
\]
Then $b_{0}\leq1$. We only need to prove $b_{0}=0$.

Suppose $b_{0}>0$. By the definition of $b_{0}$ and Lemma \ref{lemma-16}, we have

\begin{itemize}
\item[(i)] for any $p\in(0,\frac{1}{2}]$ and $\tau=ib_{0}$, equation
(\ref{i-69}) has three real roots $X$'s.

\item[(ii)] there exists $p_{0}\in(0,\frac{1}{2})$ such that equation
(\ref{i-69}) with $\tau=ib_{0}$ and $p=p_{0}$ has a multiple root $X_{0}$ with
mulitplicity $m\in \{2,3\}$.
\end{itemize}

Now we fix $\tau=ib_{0}$ and define%
\[
H(X;p):=X^{3}-3\wp(p|\tau_{0})X^{2}+\frac{g_{2}(\tau_{0})}{4}X+\frac
{2g_{3}(\tau_{0})+\wp(p|\tau_{0})g_{2}(\tau_{0})}{4}.
\]
Then $H(X_{0};p_{0})=0$, so%
\begin{align}
H(X;p)  &  =\frac{1}{m!}\frac{\partial^{m}H}{\partial X^{m}}(X_{0}%
;p_{0})(X-X_{0})^{m}\nonumber \\
&  +\frac{\partial H}{\partial p}(X_{0};p_{0})(p-p_{0})+\text{higher order
terms.} \label{i-70}%
\end{align}
If $\frac{\partial H}{\partial p}(X_{0};p_{0})\not =0$, then (\ref{i-70})
implies that there exists $p\in(0,1/2)$ satisfying $|p-p_{0}|>0$ sufficiently
small such that $H(X;p)=0$ has roots $X$'s in $\mathbb{C}\backslash \mathbb{R}%
$, a contradiction with (i). Thus $\frac{\partial H}{\partial p}(X_{0}%
;p_{0})=0$, i.e.%
\begin{equation}
12X_{0}^{2}-g_{2}(\tau_{0})=0. \label{i-71}%
\end{equation}
This together with $H(X_{0};p_{0})=0$ gives%
\begin{equation}
4X_{0}^{3}+g_{2}(\tau_{0})X_{0}+2g_{3}(\tau_{0})=0. \label{i-72}%
\end{equation}
However, (\ref{i-71})-(\ref{i-72}) leads to $g_{2}(\tau_{0})^{3}-27g_{3}%
(\tau_{0})^{2}=0$, a contradiction. Therefore, $b_{0}=0$, which completes the
proof of the case $p\in(0,\frac{1}{2}]$.

The case $p\in(0,\frac{\tau}{2}]$ can be proved in a similar way, and we omit
the details here.
\end{proof}

We conclude this section by proving Theorem \ref{THM-13}.

\begin{proof}
[Proof of Theorem \ref{THM-13}]Fix $\tau \in i\mathbb{R}_{>0}$. Then $\eta
_{1}(\tau)\in \mathbb{R}$ and $\eta_{2}(\tau)\in i\mathbb{R}$. So
$X=\frac{C\eta_{1}(\tau)-\eta_{2}(\tau)}{\tau-C}\in \mathbb{R}$ if and only if
$C\in i\mathbb{R\cup \{ \infty \}}\backslash \{ \tau \}$. Therefore, Theorem
\ref{THM-13} follows directly from Lemma \ref{lemma-17}.
\end{proof}

\end{document}